\documentclass[11pt,a4paper]{article}
\usepackage[american]{babel}
\usepackage{a4}
\usepackage{hyperref}
\usepackage{amsmath, amssymb, amsthm}
\usepackage{subfig}
\usepackage{ifthen}
\usepackage{tikz}
\usetikzlibrary{positioning,arrows,patterns}
\usepackage{tkz-euclide}
\usetkzobj{all}
\usepackage{dsfont}

\DeclareMathOperator{\re}{Re}
\DeclareMathOperator{\im}{Im}
\DeclareMathOperator{\arccot}{arccot}

\newcommand{\mC}{{\mathds C}}
\newcommand{\mR}{{\mathds R}}
\newcommand{\mZ}{{\mathds Z}}

\newtheorem{theorem}{Theorem}[section]
\newtheorem{proposition}[theorem]{Proposition}
\newtheorem{lemma}[theorem]{Lemma}
\newtheorem{corollary}[theorem]{Corollary}

\theoremstyle{definition}

\newtheorem*{definition}{Definition}
\newtheorem*{example}{Example}
\newtheorem*{remark}{Remark}

\begin{document}

\title{Discrete Riemann surfaces based on quadrilateral cellular decompositions}

\author{Alexander I.~Bobenko\footnote{Institut f\"ur Mathematik, MA 8-3, Technische
Universit\"at Berlin, Stra{\ss}e des 17. Juni 136, 10623 Berlin, Germany.}$\;^{,1}$ \and Felix G\"unther$^{*,}$\footnote{European Post-Doctoral Institute for Mathematical Sciences: Institut des Hautes \'Etudes Scientifiques, 35 route de Chartres, 91440 Bures-sur-Yvette, France; Isaac Newton Institute for Mathematical Sciences, 20 Clarkson Road, Cambridge CB30EH, United Kingdom; Erwin Schr\"odinger International Institute for Mathematical Physics, Boltzmanngasse 9, 1090 Vienna, Austria}$\;^{,2}$}

\date{}
\maketitle
\footnotetext[1]{E-mail: bobenko@math.tu-berlin.de}
\footnotetext[2]{E-mail: fguenth@math.tu-berlin.de}

\begin{abstract}
\noindent
Our aim in this paper is to provide a theory of discrete Riemann surfaces based on quadrilateral cellular decompositions of Riemann surfaces together with their complex structure encoded by complex weights. Previous work, in particular of Mercat, mainly focused on real weights corresponding to quadrilateral cells having orthogonal diagonals. We discuss discrete coverings, discrete exterior calculus, and discrete Abelian differentials. Our presentation includes several new notions and results such as branched coverings of discrete Riemann surfaces, the discrete Riemann-Hurwitz Formula, double poles of discrete one-forms and double values of discrete meromorphic functions that enter the discrete Riemann-Roch Theorem, and a discrete Abel-Jacobi map.\\ \vspace{0.5ex}

\vspace{0.5ex}

\end{abstract}

\raggedbottom
\setlength{\parindent}{0pt}
\setlength{\parskip}{1ex}


\section{Introduction}\label{sec:intro}

Linear theories of discrete complex analysis look back on a long and varied history. We refer here to the survey of Smirnov \cite{Sm10S}. Already Kirchhoff's circuit laws describe a discrete harmonicity condition for the potential function whose gradient describes the current flowing through the electric network. Discrete harmonic functions on the square lattice were studied by a number of authors in the 1920s, including Courant, Friedrichs, and Lewy \cite{CoFrLe28}. Two different notions for discrete holomorphicity on the square lattice were suggested by Isaacs \cite{Is41}. Dynnikov and Novikov studied a notion equivalent to one of them on triangular lattices in \cite{DN03}; the other was reintroduced by Lelong-Ferrand \cite{Fe44,Fe55} and Duffin \cite{Du56}. Duffin also extended the theory to rhombic lattices \cite{Du68}. Mercat \cite{Me01}, Kenyon \cite{Ke02}, and Chelkak and Smirnov \cite{ChSm11} resumed the investigation of discrete complex analysis on rhombic lattices or, equivalently, isoradial graphs.

Some two-dimensional discrete models in statistical physics exhibit conformally invariant properties in the thermodynamical limit. Such conformally invariant properties were established by Smirnov for site percolation on a triangular grid \cite{Sm01} and for the random cluster model \cite{Sm10}, by Chelkak and Smirnov for the Ising model \cite{ChSm12}, and by Kenyon for the dimer model on a square grid (domino tiling) \cite{Ke00}. In all cases, linear theories of discrete analytic functions on regular grids were important. The motivation for linear theories of discrete Riemann surfaces also comes from statistical physics, in particular, the Ising model. Mercat defined a discrete Dirac operator and discrete spin structures on quadrilateral decompositions and he identified criticality in the Ising model with rhombic quad-graphs \cite{Me01}. In \cite{Ci12}, Cimasoni discussed discrete Dirac operators and discrete spin structures on an arbitrary weighted graph isoradially embedded in a flat surface with conical singularities and how they can be used to give an explicit formula of the partition function of the dimer model on that graph. Also, he studied discrete spinors and their connection to s-holomorphic functions \cite{Ci15} that played an instrumental role in the universality results of Chelkak and Smirnov \cite{ChSm12}.

Important non-linear discrete theories of complex analysis involve circle packings or, more generally, circle patterns. Stephenson explains the links between circle packings and Riemann surfaces in \cite{Ste05}. Rodin and Sullivan proved that the Riemann mapping of a complex domain to the unit disk can be approximated by circle packings \cite{RSul87}. A similar result for isoradial circle patterns, even with irregular combinatorics, is due to B\"ucking \cite{Bue08}. In \cite{BoMeSu05} it was shown that discrete holomorphic functions describe infinitesimal deformations of circle patterns.

Mercat extended the linear theory from domains in the complex plane to discrete Riemann surfaces \cite{Me01b,Me01,Me07,Me08}. There, the discrete complex structure on a bipartite cellular decomposition of the surface into quadrilaterals is given by complex numbers $\rho_Q$ with positive real part. More precisely, the discrete complex structure defines discrete holomorphic functions by demanding that \[f(w_+)-f(w_-)=i\rho_Q\left(f(b_+)-f(b_-)\right)\] holds on any quadrilateral $Q$ with black vertices $b_-,b_+$ and white vertices $w_-,w_+$. Mercat focused on discrete complex structures given by real numbers in \cite{Me01b,Me01,Me07} and sketched some notions for complex $\rho_Q$ in \cite{Me08}. He introduced discrete period matrices \cite{Me01b, Me07}, their convergence to their continuous counterparts was shown in \cite{BoSk12}. In \cite{BoSk12}, also a discrete Riemann-Roch theorem was provided. Graph-theoretic analogues of the classical Riemann-Roch theorem and Abel-Jacobi theory were given by Baker and Norine \cite{BaNo07}.

A different linear theory for discrete complex analysis on triangulated surfaces using holomorphic cochains was introduced by Wilson \cite{Wi08}. Convergence of period matrices in that discretization to their smooth counterparts was also shown. A nonlinear theory of discrete conformality that discretizes the notion of conformally equivalent metrics was developed in \cite{BoPSp10}.

In \cite{BoG15}, a medial graph approach to discrete complex analysis on planar quad-graphs was suggested. Many results such as discrete Cauchy's integral formulae relied on discrete Stokes' Theorem~\ref{th:stokes} and Theorem~\ref{th:derivation} stating that the discrete exterior derivative is a derivation of the discrete wedge-product. These theorems turn out to be quite useful also in the current setting of discrete Riemann surfaces.

Our treatment of discrete differential forms on bipartite quad-decompositions of Riemann surfaces is close to what Mercat proposed in \cite{Me01b,Me01,Me07,Me08}. However, our version of discrete exterior calculus is based on the medial graph representation and is slightly more general. The goal of this paper is to present a comprehensive theory of discrete Riemann surfaces with complex weights $\rho_Q$ including discrete coverings, discrete exterior calculus, and discrete Abelian differentials. It includes several new notions and results including branched coverings of discrete Riemann surfaces, the discrete Riemann-Hurwitz Formula~\ref{th:Riemann_Hurwitz}, double poles of discrete one-forms and double values of discrete meromorphic functions that enter the discrete Riemann-Roch Theorem~\ref{th:Riemann_Roch}, and a discrete Abel-Jacobi map whose components are discrete holomorphic by Proposition~\ref{prop:Abel_holomorphic}.

The precise definition of a discrete complex structure will be given in Section~\ref{sec:basic}. Note that not all discrete Riemann surfaces can be realized as piecewise planar quad-surfaces, but these given by real weights $\rho_Q$ can, see Theorem~\ref{th:realization}. In Section~\ref{sec:holomap}, an idea how branch points of higher order can be modeled on discrete Riemann surfaces and a discretization of the Riemann-Hurwitz formula are given.

Since the notion of discrete holomorphic mappings developed in Section~\ref{sec:holomap} is too rigid to go further, we concentrate on discrete meromorphic functions and discrete one-forms. First, we shortly comment how the version of discrete exterior calculus developed in \cite{BoG15} generalizes to discrete Riemann surfaces in Section~\ref{sec:differentials}. The results of \cite{BoG15} that are relevant for the sequel are just stated, their proofs can be literally translated from \cite{BoG15}. Sometimes, we require in addition to a discrete complex structure local charts around the vertices of the quad-decomposition. Their existence is ensured by Proposition~\ref{prop:charts}. However, the definitions actually do not depend on the choice of charts.

In Section~\ref{sec:periods}, periods of discrete differentials are introduced and the discrete Riemann Bilinear Identity~\ref{th:RBI} is proven more or less in the same way as in the classical theory. Then, discrete harmonic differentials are studied in Section~\ref{sec:harmonic_holomorphic}. In Section~\ref{sec:Abelian_theory}, we recover the discrete period matrices of Mercat \cite{Me01b, Me07} and the discrete Abelian differentials of the first and the third kind of \cite{BoSk12} in the general setup of complex weights. Furthermore, discrete Abelian differentials of the second kind are defined. This leads to a slightly more general version of the discrete Riemann-Roch Theorem~\ref{th:Riemann_Roch}. Finally, discrete Abel-Jacobi maps and analogies to the classical theory are discussed in Section~\ref{sec:Abel}.


\section{Basic definitions}\label{sec:basic}

The aim of this section is to introduce discrete Riemann surfaces in Section~\ref{sec:setup}, giving piecewise planar quad-surfaces as an example in Section~\ref{sec:polyhedral}. There, we also discuss the question whether conversely discrete Riemann surfaces can be realized as piecewise planar quad-surfaces. The basic definitions are very similar to the notions in \cite{BoG15}, such as the medial graph introduced in Section~\ref{sec:medial}.


\subsection{Discrete Riemann surfaces}\label{sec:setup}

\begin{definition}
Let $\Sigma$ be a connected oriented surface without boundary, for short \textit{surface}. A \textit{bipartite quad-decomposition} $\Lambda$ \textit{of} $\Sigma$ is a strongly regular and locally finite cellular decomposition of $\Sigma$ such that all its 2-cells are quadrilaterals and its 1-skeleton is bipartite. Strong regularity requires that two different faces are either disjoint or share only one vertex or share only one edge; local finiteness requires that a compact subset of $\Sigma$ contains only finitely many quadrilaterals. If $\Sigma=\mC$ and $\Lambda$ is embedded in the complex plane such that all edges are straight line segments, then $\Lambda$ is called a \textit{planar quad-graph}.

Let $V(\Lambda)$ denote the set of 0-cells (\textit{vertices}), $E(\Lambda)$ the set of 1-cells (\textit{edges}), and $F(\Lambda)$ the set of 2-cells (\textit{faces} or \textit{quadrilaterals}) of $\Lambda$.
\end{definition}

In what follows, let $\Lambda$ be a bipartite quad-decomposition of the surface $\Sigma$.

\begin{definition}
We fix one decomposition of $V(\Lambda)$ into two independent sets and refer to the vertices of this decomposition as \textit{black} and \textit{white} vertices, respectively. Let $\Gamma$ be the graph defined on the black vertices where $vv'$ is an edge of $\Gamma$ if and only if its two black endpoints are vertices of a single face of $\Lambda$. Its dual graph $\Gamma^*$ is defined as the correponding graph on white vertices.
\end{definition}

\begin{remark}
The assumption of strong regularity guarantees that any edge of $\Gamma$ or $\Gamma^*$ is the diagonal of exactly one quadrilateral of $\Lambda$.
\end{remark}

\begin{definition}
$\Diamond:=\Lambda^*$ is the dual graph of $\Lambda$.
\end{definition}

\begin{definition}
If a vertex $v \in V(\Lambda)$ is a vertex of a quadrilateral $Q\in F(\Lambda)\cong V(\Diamond)$, then we write $Q \sim v$ or $v \sim Q$ and say that $v$ and $Q$ are \textit{incident} to each other.
\end{definition}

All faces of $\Lambda$ inherit an orientation from $\Sigma$. We may assume that the orientation on $\Sigma$ is chosen in such a way that the image of any orientation-preserving embedding of a two-cell $Q \in F(\Lambda)$ into the complex plane is positively oriented.

\begin{definition}
Let $Q \in F(\Lambda)$ with vertices $b_-,w_-,b_+,w_+$ in counterclockwise order, where $b_\pm \in V(\Gamma)$ and $w_\pm \in V(\Gamma^*)$. An orientation-preserving embedding $z_Q$ of $Q$ to a rectilinear quadrilateral in $\mC$ without self-intersections such that the image points of $b_-,w_-,b_+,w_+$ are vertices of the quadrilateral is called a \textit{chart} of $Q$. Two such charts are called \textit{compatible} if the oriented diagonals of the image quadrilaterals are in the same complex ratio \[\rho_Q:= -i \frac{w_+-w_-}{b_+-b_-}.\] Moreover, let $\varphi_Q:=\arccos\left(\re\left(i\rho_Q/|\rho_Q|\right)\right)$ be the angle under which the diagonal lines of $Q$ intersect.
\end{definition}

Note that $0<\varphi_Q<\pi$. Figure~\ref{fig:quadgraph} shows part of a planar quad-graph together with the notations we use for a single face $Q$ and the \textit{star of a vertex} $v$, i.e., the set of all faces incident to $v$.

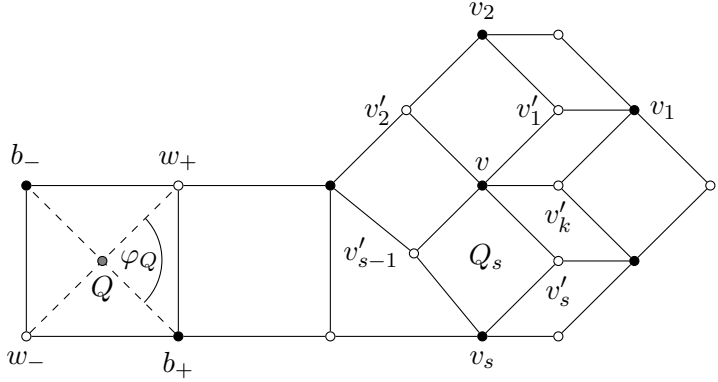
\begin{figure}[htbp]
\begin{center}
\beginpgfgraphicnamed{quad}
\begin{tikzpicture}
[white/.style={circle,draw=black,fill=white,thin,inner sep=0pt,minimum size=1.2mm},
black/.style={circle,draw=black,fill=black,thin,inner sep=0pt,minimum size=1.2mm},
gray/.style={circle,draw=black,fill=gray,thin,inner sep=0pt,minimum size=1.2mm}]
\node[white] (w1) [label=left:$v'_{s-1}$]
at (-0.9,-0.9) {};
\node[white] (w2) [label=below:$v'_s$]
 at (1,-1) {};
\node[white] (w3) [label=below:$v'_k$]
 at (1,0) {};
\node[white] (w4) [label=left:$v'_1$]
 at (1,1) {};
\node[white] (w5) [label=left:$v'_2$]
 at (-1,1) {};
\node[white] (w6) [label=above:$w_+$]
 at (-4,0) {};
\node[white] (w7)
 at (1,-2) {};
\node[white] (w8)
 at (3,0) {};
\node[white] (w9)
 at (1,2) {};
\node[white] (w10) [label=below:$w_-$]
 at (-6,-2) {};
\node[white] (w11)
 at (-2,-2) {};
\node[black] (b1) [label=above:$v$]
 at (0,0) {};
\node[black] (b2) [label=right:$v_1$]
 at (2,1) {};
\node[black] (b3)
 at (2,-1) {};
\node[black] (b4) [label=below:$v_s$]
 at (0,-2) {};
\node[black] (b5) [label=below:$b_+$]
 at (-4,-2) {};
\node[black] (b6) [label=above:$v_2$]
 at (0,2) {};
\node[black] (b7) [label=above :$b_-$]
 at (-6,0) {};
\node[black] (b8)
 at (-2,0) {};
\draw[color=white] (w1) --node[midway,color=black] {$Q_s$} (w2);
\draw (b2) -- (w9) -- (b6) -- (w5) -- (b8);
\draw (b4) -- (w7) -- (b3) -- (w8) -- (b2);
\draw (b8) -- (w1) --  (b4) --  (w2) -- (b3) -- (w3) -- (b2) -- (w4) -- (b6);
\draw (b1) -- (w1);
\draw (b1) -- (w2);
\draw (b1) -- (w3);
\draw (b1) -- (w4);
\draw (b1) -- (w5);
\draw (b8) -- (w11) -- (b4);
\draw (w6) -- (b5) -- (w10) -- (b7) -- (w6);
\draw (b5) -- (w11);
\draw (w6) -- (b8);
\draw[dashed] (b5) -- (b7);
\draw[dashed] (w10) -- (w6);
\node[gray] (z) [label=below:$Q$]
at (-5,-1) {};
\draw (-4.45,-1.55) arc (-45:43:0.8cm);
\coordinate[label=right:$\varphi_Q$] (phi) at (-4.9,-1);
\end{tikzpicture}
\endpgfgraphicnamed
\caption{Bipartite quad-decomposition with notations}
\label{fig:quadgraph}
\end{center}
\end{figure}

In addition, we denote by $\Diamond_0$ always a connected subgraph of $\Diamond$ and by $V(\Diamond_0)\subseteq V(\Diamond)$ the corresponding subset of faces of $\Lambda$. Through our identification $V(\Diamond)\cong F(\Lambda)$, we can call the elements of $V(\Diamond)$ quadrilaterals and identify them with the corresponding faces of $\Lambda$.

In particular, an equivalence class of charts $z_Q$ of a single quadrilateral $Q$ is uniquely characterized by the complex number $\rho_Q$ with a positive real part. An assignment of positive real numbers $\rho_Q$ to all faces $Q$ of $\Lambda$ was the definition of a discrete complex structure Mercat used in \cite{Me01}. In his subsequent work \cite{Me08}, he proposed a generalization to complex $\rho_Q$ with positive real part. Mercat's notion of a discrete Riemann surface is equivalent to the definition we give:

\begin{definition}
A \textit{discrete Riemann surface} is a triple $(\Sigma, \Lambda, z)$ of a bipartite quad-decomposition $\Lambda$ of a surface $\Sigma$ together with an \textit{atlas} $z$, i.e., a collection of charts $z_Q$ of all faces $Q \in F(\Lambda)$ that are compatible to each other. An assignment of complex numbers $\rho_Q$ with positive real part to the faces $Q$ of the quad-decomposition is said to be a \textit{discrete complex structure}.

$(\Sigma, \Lambda, z)$ is said to be \textit{compact} if the surface $\Sigma$ is compact.
\end{definition}

Note that real $\rho_Q$ correspond to quadrilaterals $Q$ whose diagonals are orthogonal to each other. They arise naturally if one considers a Delaunay triangulation of a polyhedral surface $\Sigma$ and places the vertices of the dual at the circumcenters of the triangles. Discrete Riemann surfaces based on this structure were investigated in \cite{BoSk12}. There, the above definition of a discrete Riemann surface was suggested as a generalization.

\begin{remark}
Compared to the classical theory, charts around vertices of $\Lambda$ are missing so far and were not considered by previous authors. In order to obtain definitions that can be immediately motivated from the classical theory, we will introduce such charts in our setting. However, we do not include them in the definition of a discrete Riemann surface. As it turns out, there always exist appropriate charts around vertices and besides discrete derivatives of functions on $V(\Diamond)$ all of our notions do not depend on these charts.
\end{remark}

\begin{definition}
Let $v \in V(\Lambda)$. An orientation-preserving embedding $z_v$ of the star of $v$ to the star of a vertex of a planar quad-graph $\Lambda'$ that maps vertices of $\Lambda$ to vertices of $\Lambda'$ is said to be a \textit{chart} as well. $z_v$ is said to be \textit{compatible} with the discrete complex structure of the discrete Riemann surface $(\Sigma,\Lambda,z)$ if for any quadrilateral $Q\sim v$ the induced chart of $z_v$ on $Q$ is compatible to $z_Q$.
\end{definition}
\begin{remark}
When we later speak about particular charts $z_v$, we always refer to charts compatible with the discrete complex structure.
\end{remark}

\begin{proposition}\label{prop:charts}
Let $\Lambda$ be a bipartite quad-decomposition of a Riemann surface $\Sigma$, and let the numbers $\rho_Q$, $Q\in V(\Diamond)$, define a discrete complex structure. Then, there exists an atlas $z$ such that the image quadrilaterals of charts $z_Q$ are parallelograms with the oriented ratio of diagonals equal to $i\rho_Q$ and such that for any $v \in V(\Lambda)$ there exists a chart $z_v$ compatible with the discrete complex structure.
\end{proposition}
\begin{proof}
The construction of the charts $z_Q$ is simple: In the complex plane, the quadrilateral with black vertices $\pm 1$ and white vertices $\pm i\rho_Q$ is a parallelogram with the desired oriented ratio of diagonals.

In contrast, the construction of charts $z_v$ is more delicate. See Figure~\ref{fig:construction} for a visualization. Let us consider the star of a vertex $v\in V(\Gamma)$. If $v$ is white, then just replace $\rho_Q$ by $1/\rho_Q$ and $\varphi_Q$ by $\pi-\varphi_Q$ in the following. Let $Q_1,Q_2,\ldots,Q_k$ be the quadrilaterals incident to $v$.

We choose $0<\theta<\pi$ in such a way that $\theta<\varphi_{Q_1}<\pi-\theta$. Let $\alpha_1:= \pi-\theta$, and define $\alpha_s:= (\pi+\theta)/(k-1)$ for the other $s$. Then, all $\alpha_s$ sum up to $2\pi$.

First, we construct the images of $Q_s$, $s\neq 1$, starting with an auxiliary construction. As in Figure~\ref{fig:quadgraph}, let $v$, $v'_{s-1}$, $v_s$, $v'_{s}$ be the vertices of $Q_s$ in counterclockwise order. Then, we map $v'_{s-1}$ to $-1$ and $v'_{s}$ to $1$. All points $x$ that enclose an oriented angle $\alpha_s>0$ with $\pm 1$ lie on a circular arc above the real axis. Since the real part of $\rho_{Q_s}$ is positive, the ray $ti\rho_{Q_s}$, $t>0$, intersects this arc in exactly one point. If we choose the intersection point $x_v$ as the image of $v$ and $x_{v_s}:=x_v-2i\rho_{Q_s}$ as the image of $v_s$,then we get a quadrilateral in $\mC$ that has the desired oriented ratio of diagonals $i\rho_{Q_s}$. The quadrilateral is convex if and only if $x_v-2i\rho_{Q_s}$ has nonpositive imaginary part.

Now, we translate all the image quadrilaterals such that $v$ is always mapped to zero. By construction, the image of $Q_s$ is contained in a cone of angle $\alpha_s$. Thus, we can rotate and scale the images of $Q_s$, $s\neq 1$, in such a way that they do not overlap and that the images of edges $vv'_{s}$ coincide. Since all $\alpha_s$ sum up to $2\pi$, there is still a cone of angle $\alpha_1=\pi-\theta$ empty.

Let us identify the vertices $v$, $v'_k$, and $v'_1$ with their corresponding images and choose $q$ on the line segment $v'_kv'_1$. If $q$ approaches the vertex $v'_k$, then $\angle vqv'_k \to \pi - \angle v'_1v'_kv$, and if $q$ approaches $v'_1$, then $\angle vqv'_k \to \angle vv'_1v'_k$. Since \[\angle vv'_1v'_k<\theta<\varphi_{Q_1}<\pi-\theta<\pi - \angle v'_1v'_kv,\] there is a point $q$ on the line segment such that $\angle vqv'_k = \varphi_Q$. If we take the image of $v_1$ on the ray $tq$, $t>0$, such that its distance to the origin is $|v'_k-v'_1|/|\rho_{Q_1}|$, then we obtain a quadrilateral with the oriented ratio of diagonals $i\rho_{Q_1}$.
\end{proof}

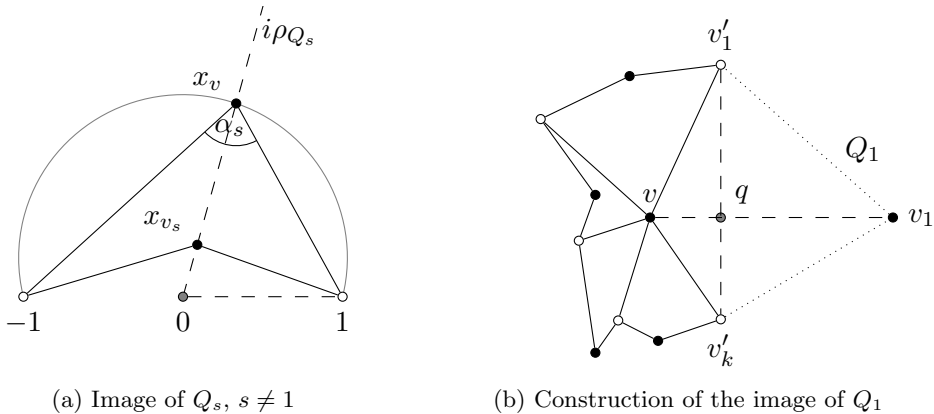
\begin{figure}[htbp]
   \centering
    \subfloat[Image of $Q_s$, $s\neq 1$]{
    \beginpgfgraphicnamed{construction1}
			\begin{tikzpicture}[white/.style={circle,draw=black,fill=white,thin,inner sep=0pt,minimum size=1.2mm},
black/.style={circle,draw=black,fill=black,thin,inner sep=0pt,minimum size=1.2mm},
gray/.style={circle,draw=black,fill=gray,thin,inner sep=0pt,minimum size=1.2mm},scale=0.7]
			\clip(-3.5,-1.5) rectangle (3.5,5.5);
			\draw [domain=0.0:17, dash pattern=on 5pt off 5pt] plot(\x,{(-0.0--3.66*\x)/1.0});
			\tkzDefPoint(-3,0){w1}\tkzDefPoint(1,3.66){b1}\tkzDefPoint(3,0){w2}
			\tkzCircumCenter(w1,b1,w2)
		  \tkzGetPoint{O}
			\tkzDrawArc(O,w2)(w1)
			\node[white] (w1) [label=below:$-1$] at (-3,0) {};
			\node[white] (w2) [label=below:$1$] at (3,0) {};
			\node[gray] (x) [label=below:$0$] at (0,0) {};
			\node[black] (b1) [label=above left:$x_v$] at (1,3.66) {};
			\node[black] (b2) [label=above left:$x_{v_s}$] at (0.2686098530106421,0.9831120620189502) {};
			\draw [dash pattern=on 5pt off 5pt] (x) -- (w2);
			\draw (w1)--(b1)--(w2)--(b2)--(w1);
			\draw (2,5) node {$i\rho_{Q_s}$};
			\draw (0.41,3.12) arc (222.46:298.65:0.8cm);
      \coordinate[label=right:$\alpha_{s}$] (alpha) at (0.4,3.2);
		\end{tikzpicture}
		\endpgfgraphicnamed}
		\qquad
		\subfloat[Construction of the image of $Q_1$]{
		\beginpgfgraphicnamed{construction2}
		\begin{tikzpicture}
		[white/.style={circle,draw=black,fill=white,thin,inner sep=0pt,minimum size=1.2mm},
black/.style={circle,draw=black,fill=black,thin,inner sep=0pt,minimum size=1.2mm},
gray/.style={circle,draw=black,fill=gray,thin,inner sep=0pt,minimum size=1.2mm},scale=0.75]
				\clip(2,-4) rectangle (10.8,2.3);
					\node[black] (b1) [label=right:$v_1$] at (9.82,-1.2) {};
					\node[black] (b2) [label=above:$v$] at (5.56,-1.2) {};
					\node[black] (b3) at (4.6,-0.8) {};
					\node[black] (b4) at (5.7,-3.38) {};
					\node[black] (b5) at (4.6,-3.59) {};
					\node[black] (b6) at (5.2,1.3) {};
					\node[white] (w1) [label=above:$v'_1$] at (6.8,1.5) {};
					\node[white] (w2) [label=below:$v'_{k}$] at (6.8,-3) {};
					\node[white] (w3) at (3.64,0.54) {};
					\node[white] (w4) at (5,-3.02) {};
					\node[white] (w5) at (4.31,-1.61) {};
					\node[gray] (q) [label=above right:$q$] at (6.8,-1.2) {};
					\coordinate[label=right:$Q_1$] (q1) at (8.8,0);
				\draw [dash pattern=on 5pt off 5pt] (b2)-- (b1);
				\draw [dash pattern=on 5pt off 5pt] (w1)-- (w2);
				\draw [dotted] (w2)-- (b1)--(w1);
				\draw (b2) -- (w1);
				\draw (b2) -- (w2);
				\draw (b2) -- (w3);
				\draw (b2) -- (w4);
				\draw (b2) -- (w5);
				\draw (w1) -- (b6) -- (w3);
				\draw (w3) -- (b3) -- (w5);
				\draw (w5) -- (b5) -- (w4);
				\draw (w4) -- (b4) -- (w2);	
			\end{tikzpicture}
		\endpgfgraphicnamed}
   \caption[]{Visualization of the proof of Proposition~\ref{prop:charts}}
   \label{fig:construction}
\end{figure}

\begin{remark}
Note that dependent on the discrete Riemann surface it could be impossible to find charts around vertex stars whose images consist of convex quadrilaterals only. Indeed, the interior angle at a black vertex $v$ of a convex quadrilateral with purely imaginary oriented ratio of diagonals $i\rho$ has to be at least $\arctan(|\rho|)=\pi/2-\arccot(|\rho|)$. In particular, the interior angles at $v$ of five or more incident convex quadrilaterals $Q_s$ such that $\rho_{Q_s}>\pi$ sum up to more than $2\pi$.
\end{remark}


\subsection{Piecewise planar quad-surfaces and discrete Riemann surfaces}\label{sec:polyhedral}

A polyhedral surface $\Sigma$ without boundary consists of Euclidean polygons glued together along common edges. Clearly, there are a lot of possibilities to make it a discrete Riemann surface. An essentially unique way to make a closed polyhedral surface a discrete Riemann surface is the following (see for example \cite{BoSk12}): The vertices of the (essentially unique) Delaunay triangulation are the black vertices and the centers of the circumcenters of the triangles are the white vertices (Figure~\ref{fig:DelaunayVoronoi}). The corresponding quadrilaterals possess isometric embeddings into the complex plane and form together a discrete Riemann surface. Note that all quadrilaterals are kites, corresponding to a discrete complex structure with real numbers $\rho_Q$ that are given by the so-called \textit{cotangent weights} \cite{PP93}. The corresponding cellular decomposition is called \textit{Delaunay-Voronoi quadrangulation}.

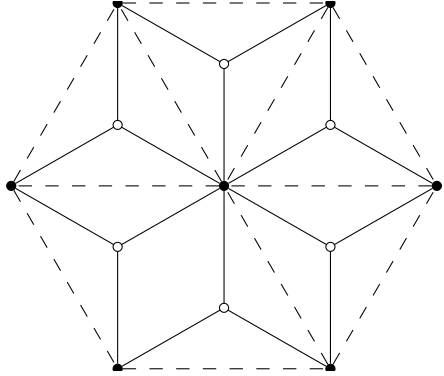
\begin{figure}[htbp]
\begin{center}
\beginpgfgraphicnamed{DelaunayVoronoi}
\begin{tikzpicture}
[white/.style={circle,draw=black,fill=white,thin,inner sep=0pt,minimum size=1.2mm},
black/.style={circle,draw=black,fill=black,thin,inner sep=0pt,minimum size=1.2mm},
gray/.style={circle,draw=black,fill=gray,thin,inner sep=0pt,minimum size=1.2mm},scale=0.7]
\clip(-4.1,-3.5) rectangle (4.1,3.5);
\node[black] (b1)
 at (-2,3.464) {};
\node[black] (b2)
 at (2,3.464) {};
\node[black] (b3)
 at (-4,0) {};
\node[black] (b4)
 at (0,0) {};
\node[black] (b5)
 at (4,0) {};
\node[black] (b6)
 at (-2,-3.464) {};
\node[black] (b7)
 at (2,-3.464) {};
\tkzSetUpPoint[fill=white,size=3mm]
\tkzCircumCenter(b1,b2,b4) \tkzGetPoint{w2}
\tkzCircumCenter(b1,b3,b4) \tkzGetPoint{w1}
\tkzCircumCenter(b5,b2,b4) \tkzGetPoint{w3}
\tkzCircumCenter(b3,b6,b4) \tkzGetPoint{w4}
\tkzCircumCenter(b6,b7,b4) \tkzGetPoint{w5}
\tkzCircumCenter(b7,b5,b4) \tkzGetPoint{w6}
\draw [dash pattern=on 5pt off 5pt] (b3)--(b1)-- (b2)-- (b5)-- (b4)-- (b3)-- (b6)-- (b7)-- (b4)--(b1);
\draw [dash pattern=on 5pt off 5pt] (b7)--(b5);
\draw [dash pattern=on 5pt off 5pt] (b2)--(b4);
\draw (w1)--(b1);
\draw (w1)--(b3);
\draw (w1)--(b4);
\draw (w2)--(b1);
\draw (w2)--(b2);
\draw (w2)--(b4);
\draw (w3)--(b5);
\draw (w3)--(b2);
\draw (w3)--(b4);
\draw (w4)--(b3);
\draw (w4)--(b6);
\draw (w4)--(b4);
\draw (w5)--(b6);
\draw (w5)--(b7);
\draw (w5)--(b4);
\draw (w6)--(b7);
\draw (w6)--(b5);
\draw (w6)--(b4);
\tkzDrawPoints(w1,w2,w3,w4,w5,w6)
\end{tikzpicture}
\endpgfgraphicnamed
\caption{Delaunay-Voronoi quadrangulation corresponding to a Delaunay triangulation}
\label{fig:DelaunayVoronoi}
\end{center}
\end{figure}

Let us suppose that the polyhedral surface $\Sigma$ is a piecewise planar quad-surface. Then, $\Sigma$ becomes a discrete Riemann surface in a canonical way. In the classical theory, any polyhedral surface possesses a canonical complex structure and any compact Riemann surface can be recovered from some polyhedral surface \cite{Bost92}. In the discrete setting, the situation is different.

\begin{theorem}\label{th:realization}
Let $(\Sigma,\Lambda,z)$ be a compact discrete Riemann surface.
\begin{enumerate}
\item If all numbers $\rho_Q$ of the discrete complex structure are real, then there exists a polyhedral surface consisting of rhombi such that its induced discrete complex structure is the one of $(\Sigma,\Lambda,z)$.
\item If all numbers $\rho_Q$ of the discrete complex structure are real but one is not, then there exists no piecewise planar quad-surface with the combinatorics of $\Lambda$ such that its induced discrete complex structure coincides with the one of $(\Sigma,\Lambda,z)$.
\end{enumerate}
\end{theorem}
\begin{proof}
(i) The diagonals of a rhombus intersect orthogonally. Clearly, the oriented ratio of diagonals of a rhombus $Q$ is $i\rho_Q=i\tan\left(\alpha/2\right)$, where $\alpha$ denotes the interior angle at a black vertex. Choosing $\alpha=2\arctan(\rho_Q)$ gives a rhombus with the desired oriented ratio of diagonals. If all the side lengths of the rhombi are one, then we can glue them together to obtain the desired closed polyhedral surface.

(ii) For a chart $z_Q$ of $Q\in V(\Diamond)$, consider the image $z_Q(Q)$. We denote the lengths of its edges by $a,b,c,d$ in counterclockwise order, starting with an edge going from a black to a white vertex, and the lengths of the line segments connecting the vertices with the intersection of the diagonal lines by $e_1,e_2,f_1,f_2$ as in Figure~\ref{fig:cosine}.

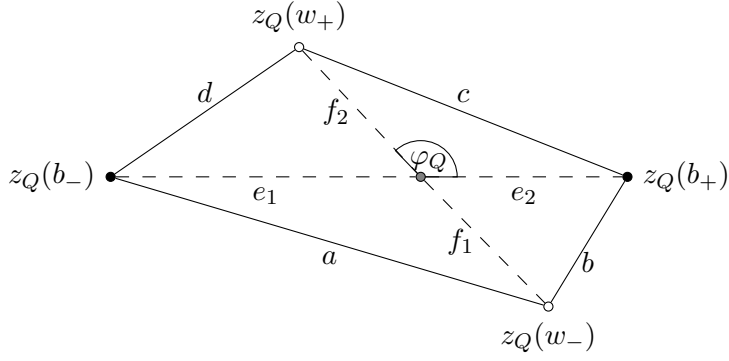
\begin{figure}[htbp]
   \centering
		\beginpgfgraphicnamed{cosine}
		\begin{tikzpicture}
		[white/.style={circle,draw=black,fill=white,thin,inner sep=0pt,minimum size=1.2mm},
black/.style={circle,draw=black,fill=black,thin,inner sep=0pt,minimum size=1.2mm},
gray/.style={circle,draw=black,fill=gray,thin,inner sep=0pt,minimum size=1.2mm},scale=4]
					\draw [shift={(1.02,0.5)}] (0,0) -- (-0.13:0.12) arc (-0.13:133.79:0.12) -- cycle;
					\node[white] (w1) [label=below:$z_Q(w_-)$] at (1.44,0.07) {};
					\node[white] (w2) [label=above:$z_Q(w_+)$] at (0.62,0.93) {};
					\node[black] (b1) [label=left:$z_Q(b_-)$] at (0,0.5) {};
					\node[black] (b2) [label=right:$z_Q(b_+)$] at (1.7,0.5) {};
					\draw (b1)--node[midway,above] {$d$} (w2);
					\draw (w2)--node[midway,above] {$c$} (b2);
					\draw (b2)--node[midway,below] {$b$} (w1);
					\draw (w1)--node[midway,below] {$a$} (b1);
					\node[gray] (q) at (1.02,0.5) {};
					\draw [dash pattern=on 5pt off 5pt] (w2)--node[midway,left] {$f_2$}(q)--node[midway,left] {$f_1$} (w1);
					\draw [dash pattern=on 5pt off 5pt] (b1)--node[midway,below] {$e_1$}(q)--node[midway,below] {$e_2$} (b2);
					\coordinate[label=right:$\varphi_Q$] (phi) at (0.95,0.55);
			\end{tikzpicture}
		\endpgfgraphicnamed
   \caption{Illustration of the formula $a^2-b^2+c^2-d^2=-2\cos(\varphi_Q)ef$}
   \label{fig:cosine}
\end{figure}

Cosine theorem implies \begin{align*}
a^2&=e_1^2+f_1^2-2e_1f_1\cos(\varphi_Q),\\
b^2&=e_2^2+f_1^2+2e_2f_1\cos(\varphi_Q),\\
c^2&=e_2^2+f_2^2-2e_2f_2\cos(\varphi_Q),\\
d^2&=e_1^2+f_2^2+2e_1f_2\cos(\varphi_Q).
\end{align*}
Taking the alternating sum, we get \[a^2-b^2+c^2-d^2=-2\cos(\varphi_Q)(e_1f_1+e_2f_1+e_2f2+e_1f_2)=-2\cos(\varphi_Q)ef,\] where $e:=e_1+e_2$ and $f:=f_1+f_2$ are the lengths of the two diagonals. In particular, $\varphi_Q=\pi/2$ if and only if $a^2-b^2+c^2-d^2=0$.

Suppose there is a piecewise planar quad-surface with the combinatorics of $\Lambda$ such that its induced discrete complex structure is the one of $(\Sigma,\Lambda,z)$. Let us orient all edges from the white to its black endpoint. For each quadrilateral $Q$ we consider its alternating sum of edge lengths such that the sign in front of an edge that is oriented in counterclockwise direction is positive and negative otherwise. If we sum these sums up for all $Q\in V(\Diamond)$, then each edge length appears twice with different signs, so the sum is zero. On the other hand, for all but one $Q$ $\rho_Q$ is real and the remaining one is not, so the sum is nonzero, contradiction. Thus, there cannot exist such a piecewise planar quad-surface.   
\end{proof}


\subsection{Medial graph} \label{sec:medial}

\begin{definition}
The \textit{medial graph} $X$ of the bipartite quad-decomposition $\Lambda$ of the surface $\Sigma$ is defined as the following cellular decomposition of $\Sigma$. Its vertex set is given by all the midpoints of edges of $\Lambda$, and two vertices $x,x'$ are adjacent if and only if the corresponding edges belong to the same face $Q$ of $\Lambda$ and have a vertex $v\in V(\Lambda)$ in common. We denote this edge (or 1-cell) by $[Q,v]$. A \textit{face} (or 2-cell) $F_v$ of $X$ corresponding to $v\in V (\Lambda)$ shall have the edges of $\Lambda$ incident to $v$ as vertices, and a \textit{face} (or 2-cell) $F_Q$ of $X$ corresponding to $Q\in F(\Lambda)\cong V(\Diamond)$ shall have the four edges of $\Lambda$ belonging to $Q$ as vertices. In Figure~\ref{fig:medial}, the vertices of the medial graph are colored gray. In this sense, the set $F(X)$ of \textit{faces} of $X$ is defined and in bijection with $V(\Lambda)\cup V(\Diamond)$.
\end{definition}

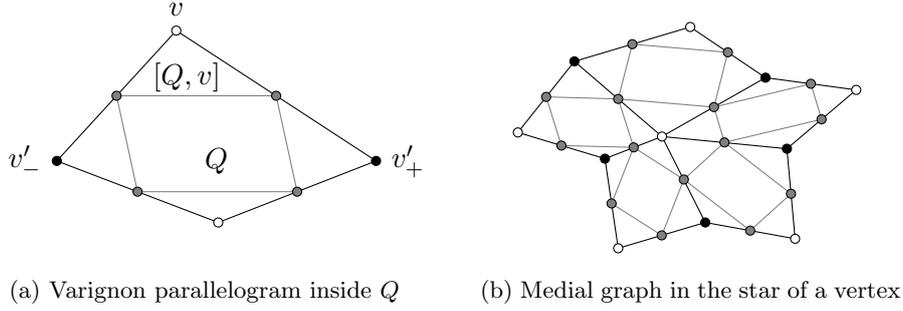
\begin{figure}[htbp]
   \centering
    \subfloat[Varignon parallelogram inside $Q$]{
    \beginpgfgraphicnamed{medial1}
			\begin{tikzpicture}[white/.style={circle,draw=black,fill=white,thin,inner sep=0pt,minimum size=1.2mm},
black/.style={circle,draw=black,fill=black,thin,inner sep=0pt,minimum size=1.2mm},
gray/.style={circle,draw=black,fill=gray,thin,inner sep=0pt,minimum size=1.2mm},scale=0.6]
			\clip(-1.7,-6.5) rectangle (7.4,-0.2);
			\draw (-0.6,-4.16)-- (2.02,-1.28);
			\draw (2.02,-1.28)-- (6.4,-4.16);
			\draw (6.4,-4.16)-- (2.94,-5.52);
			\draw (2.94,-5.52)-- (-0.6,-4.16);
			\draw [color=gray] (1.17,-4.84)-- (0.71,-2.72);
			\draw [color=gray] (0.71,-2.72)-- (4.21,-2.72);
			\draw [color=gray] (4.21,-2.72)-- (4.67,-4.84);
			\draw [color=gray] (4.67,-4.84)-- (1.17,-4.84);
			\node[white] (w1) at (2.94,-5.52) {};
			\node[white] (w2) [label=above:$v$] at (2.02,-1.28) {};
			\node[black] (b1) [label=left:$v'_{-}$] at (-0.6,-4.16) {};
			\node[black] (b2) [label=right:$v'_{+}$] at (6.4,-4.16) {};
			\draw (2.9,-4.16) node {$Q$};
			\node[gray] (m1) at (0.71,-2.72) {};
			\node[gray] (m2) at (4.21,-2.72) {};
			\node[gray] (m3) at (4.67,-4.84) {};
			\node[gray] (m4) at (1.17,-4.84) {};
			\draw (2.25,-2.3) node {$[Q,v]$};
		\end{tikzpicture}
		\endpgfgraphicnamed}
		\qquad
		\subfloat[Medial graph in the star of a vertex]{
		\beginpgfgraphicnamed{medial2}
		\begin{tikzpicture}
		[white/.style={circle,draw=black,fill=white,thin,inner sep=0pt,minimum size=1.2mm},
black/.style={circle,draw=black,fill=black,thin,inner sep=0pt,minimum size=1.2mm},
gray/.style={circle,draw=black,fill=gray,thin,inner sep=0pt,minimum size=1.2mm},scale=0.6]
	\clip(2,-4) rectangle (10.8,2);
				\draw (2.4,-1.03)-- (3.64,0.54);
				\draw (3.64,0.54)-- (6.18,1.3);
				\draw (3.64,0.54)-- (5.56,-1.12);
				\draw (5.56,-1.12)-- (4.31,-1.61);
				\draw (4.31,-1.61)-- (2.4,-1.03);
				\draw (4.31,-1.61)-- (4.6,-3.59);
				\draw (9.82,-0.09)-- (7.83,0.18);
				\draw (9.82,-0.09)-- (8.3,-1.39);
				\draw (7.83,0.18)-- (5.56,-1.12);
				\draw (8.3,-1.39)-- (5.56,-1.12);
				\draw (7.83,0.18)-- (6.18,1.3);
				\draw (5.56,-1.12)-- (6.51,-3.02);
				\draw (6.51,-3.02)-- (4.6,-3.59);
				\draw (6.51,-3.02)-- (8.48,-3.38);
				\draw (8.48,-3.38)-- (8.3,-1.39);
				\draw [color=gray] (4.94,-1.36)-- (4.6,-0.29);
				\draw [color=gray] (4.6,-0.29)-- (6.7,-0.47);
				\draw [color=gray] (6.7,-0.47)-- (6.93,-1.25);
				\draw [color=gray] (6.93,-1.25)-- (6.04,-2.07);
				\draw [color=gray] (6.04,-2.07)-- (4.94,-1.36);
					\node[white] (w1) at (9.82,-0.09) {};
					\node[white] (w2) at (5.56,-1.12) {};
					\node[white] (w3) at (2.4,-1.03) {};
					\node[white] (w4) at (8.48,-3.38) {};
					\node[white] (w5) at (4.6,-3.59) {};
					\node[white] (w6) at (6.18,1.3) {};
					\node[black] (b1) at (7.83,0.18) {};
					\node[black] (b2) at (8.3,-1.39) {};
					\node[black] (b3) at (3.64,0.54) {};
					\node[black] (b4) at (6.51,-3.02) {};
					\node[black] (b5) at (4.31,-1.61) {};
					\draw [color=gray] (7.495,-3.2) --(8.39,-2.385)--(6.93,-1.25)--(9.06,-0.74)--(8.825,0.045)--(6.7,-0.47)--(7.005,0.74) --(4.91,0.92)--(4.6,-0.29)--(3.02,-0.245)--(3.355,-1.32)--(4.94,-1.36)--(4.455,-2.6)--(5.555,-3.305)--(6.04,-2.07)--(7.495,-3.2);				
					\node[gray] (m1) at (6.93,-1.25) {};
					\node[gray] (m2) at (6.7,-0.47) {};
					\node[gray] (m3) at (4.6,-0.29) {};
					\node[gray] (m4) at (4.94,-1.36) {};
					\node[gray] (m5) at (6.04,-2.07) {};
					\node[gray] (m6) at (7.495,-3.2) {};
					\node[gray] (m7) at (8.39,-2.385) {};
					\node[gray] (m8) at (9.06,-0.74) {};
					\node[gray] (m9) at (8.825,0.045) {};
					\node[gray] (m10) at (7.005,0.74) {};
					\node[gray] (m11) at (4.91,0.92) {};
					\node[gray] (m12) at (3.02,-0.245) {};
					\node[gray] (m13) at (3.355,-1.32) {};
					\node[gray] (m14) at (4.455,-2.6) {};
					\node[gray] (m15) at (5.555,-3.305) {};
			\end{tikzpicture}
		\endpgfgraphicnamed}
   \caption[]{Medial graph $X$ and notation of its edges}
   \label{fig:medial}
\end{figure}

A priori, $X$ is just a combinatorial datum, giving a cellular decomposition of $\Sigma$ with induced orientation. But charts $z_v$ and $z_Q$ induce geometric realizations of the faces $F_v$ and $F_Q$ corresponding to $v\in V(\Lambda)$ and $Q\in V(\Diamond)$, respectively, in the complex plane. For this, we identify vertices of $X$ with the midpoints of the images of corresponding edges and map the edges of $X$ to straight line segments. $z_Q$ always induces an orientation-preserving embedding, $z_v$ does if it maps the quadrilaterals of the star of $v$ to quadrilaterals whose interior angle at $z_v(v)$ is less than $\pi$. Due to Varignon's theorem, $z_Q(F_Q)$ is a parallelogram, even if $z_Q(Q)$ is not. Also, the image of the oriented edge $e=[Q,v]$ of $X$ connecting the edges $vv'_-$ with $vv'_+$ is just half the image of the diagonal: $2z_Q(e)=z_Q(v'_+)-z_Q(v'_-)$. In this sense, $e$ is \textit{parallel} to the edge $v'_-v'_+$ of $\Gamma$ or $\Gamma^*$.

\begin{definition}
We call an edge of $X$ \textit{black} or \textit{white} if it is parallel to an edge of $\Gamma$ or $\Gamma^*$, respectively.
\end{definition}

\begin{remark}
Even if $z_v$ does not induce an orientation-preserving embedding of $F_v$, we still obtain a rectilinear polygon $z_v(F_v)$ by the construction described above. In particular, the algebraic area of $z_v(F_v)$ is defined, where the orientation of $z_v(F_v)$ is inherited from the orientation of the star of $v$ on $\Sigma$.
\end{remark}

\begin{definition}
For a connected subgraph $\Diamond_0 \subseteq \Diamond$, we denote by $\Lambda_0$ the subgraph of $\Lambda$ whose vertices and edges are exactly the vertices and edges of the quadrilaterals in $V(\Diamond_0)$. An \textit{interior} vertex $v\in V(\Lambda_0)$ is a vertex such that all incident faces in $\Lambda$ belong to $V(\Diamond_0)$. All other vertices of $\Lambda_0$ are said to be \textit{boundary vertices}. Let $\Gamma_0$ and $\Gamma_0^*$ denote the subgraphs of $\Gamma$ and of $\Gamma^*$ whose edges are exactly the diagonals of quadrilaterals in $V(\Diamond_0)$.

$\Diamond_0\subseteq\Diamond$ is said to \textit{form a simply-connected closed region} if the union of all quadrilaterals in $V(\Diamond_0)$ forms a simply-connected closed region in $\Sigma$.

Furthermore, we denote by $X_0 \subseteq X$ the connected subgraph of $X$ consisting of all edges $[Q,v]$ where $Q\in V(\Diamond_0)$ and $v$ is a vertex of $Q$. For a finite collection $F$ of faces of $X_0$, $\partial F$ denotes the union of all counterclockwise oriented boundaries of faces in $F$, where oriented edges in opposite directions cancel each other out. 
\end{definition}


\section{Discrete holomorphic mappings}\label{sec:holomap}

Throughout this section, let $(\Sigma, \Lambda, z)$ and $(\Sigma', \Lambda', z')$ be discrete Riemann surfaces.


\subsection{Discrete holomorphicity}\label{sec:Cauchy_Riemann}

The following notion of discrete holomorphic functions is essentially due to Mercat \cite{Me01,Me07,Me08}.

\begin{definition}
Let $f:V(\Lambda_0) \to \mC$. $f$ is said to be \textit{discrete holomorphic} if the \textit{discrete Cauchy-Riemann equation} \[\frac{f(b_+)-f(b_-)}{z_Q(b_+) - z_Q(b_-)}=\frac{f(w_+)-f(w_-)}{z_Q(w_+) - z_Q(w_-)}\] is satisfied for all quadrilaterals $Q \in V(\Diamond_0)$ with vertices $b_-,w_-,b_+,w_+$ in counterclockwise order, starting with a black vertex. $f$ is \textit{discrete antiholomorphic} if $\bar{f}$ is discrete holomorphic.
\end{definition}

Note that the discrete Cauchy-Riemann equation in the chart $z_Q$ is equivalent to the corresponding equation in a compatible chart $z'_Q$, i.e., it depends on the discrete complex structure only.

\begin{definition}
A mapping $f:V(\Lambda) \to V(\Lambda')$ is said to be \textit{discrete holomorphic} if the following conditions are satisfied:
\begin{enumerate}
 \item $f(V(\Gamma))\subseteq V(\Gamma')$ and $f(V(\Gamma^*))\subseteq{V({\Gamma'}^*)}$;
 \item for any quadrilateral $Q \in F(\Lambda)$, there exists a face $Q' \in F(\Lambda')$ such that $f(v)\sim Q'$ for all $v\sim Q$;
 \item for any quadrilateral $Q \in F(\Lambda)$, the function $z'_{Q'} \circ f: V(Q)\to \mC$ is discrete holomorphic.
\end{enumerate}
\end{definition}

The first condition asserts that $f$ respects the bipartite structures of the quad-decompositions. The second one discretizes continuity and guarantees that the third holomorphicity condition makes sense.

\begin{remark}
Note that a discrete holomorphic mapping $f$ may be \textit{biconstant} (constant at black and constant at white vertices) at some quadrilaterals, but not at all, whereas in the smooth case, any holomorphic mapping that is locally constant somewhere is constant on connected components. We resolve this contradiction by interpreting quadrilaterals where $f$ is biconstant as branch points.
\end{remark}


\subsection{Simple properties and branch points}\label{sec:branch}

The following lemma discretizes the classical fact that nonconstant holomorphic mappings are open.

\begin{lemma} \label{lem:open_map}
Let $f:V(\Lambda) \to V(\Lambda')$ be a discrete holomorphic mapping. Then, for any $v \in V(\Lambda)$ there exists a nonnegative integer $k$ such that the image of the star of $v$ goes $k$ times along the star of $f(v)$ (preserving the orientation).
\end{lemma}
\begin{proof}
By definition of discrete holomorphicity, the image of the star of $v$ is contained in the star of $f(v)$ and the orientation is preserved. If $f$ is biconstant around the star of $v$, then the statement is true with $k=0$. So assume that $f$ is not biconstant there. Then, at least one quadrilateral in the star is mapped to a complete quadrilateral in the star of $f(v)$. The next quadrilateral is either mapped to an edge if $f$ is biconstant at this quadrilateral or to the neighboring quadrilateral. Since this has to close up in the end, the image goes $k>0$ times along the star of $f(v)$.
\end{proof}

\begin{definition}
If the number $k$ in the lemma above is zero, then we say that $v$ is a \textit{vanishing point}. Otherwise, $v$ is a \textit{regular point}. If $k>1$, then we say that $v$ is a \textit{branch point} of multiplicity $k$. In any case, we define $b_f(v)=k-1$ as the \textit{branch number} of $f$ at $v$.

If $f$ is biconstant at $Q\in F(\Lambda)$, then we say that $Q$ is a \textit{branch point} of multiplicity two with branch number $b_f(Q)=1$. Otherwise, $Q$ is not a branch point and $b_f(Q)=0$.
\end{definition}

\begin{figure}[htbp]
   \centering
    \subfloat{
    \includegraphics[scale=0.4]{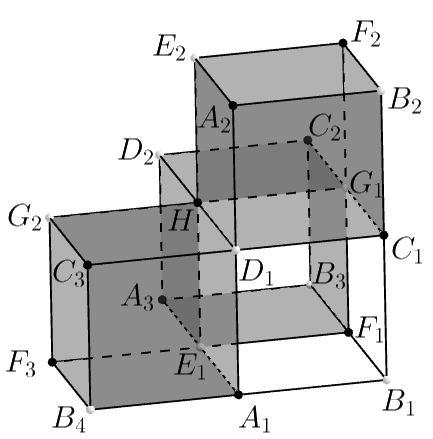}}
		\qquad
		\subfloat{
		\includegraphics[scale=0.4]{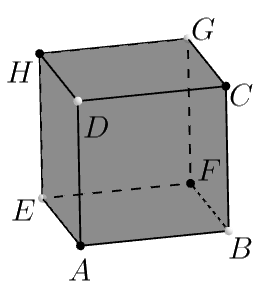}}
   \caption[]{Two-sheeted cover of a cube by a surface of genus three}
   \label{fig:cover}
\end{figure}

\begin{example}
Figure~\ref{fig:cover} shows a two-sheeted covering of an elementary cube by a surface of genus three that is composed of 8 vertices, 24 edges, and 12 faces. For this, points $X_i$ and $X_j$, $X \in \{A,B,C,D,E,F,G\}$ and $i,j \in \{1,2,3,4\}$, are identified. The mapping $f$ between these surfaces maps a point $X_i$ to the corresponding point $X$ on the cube.

The bipartite quad-decomposition of the surface of genus three is not strongly regular, but a uniform decomposition of each square into nine smaller squares gives us a strongly regular quad-decomposition. This makes both surfaces discrete Riemann surfaces in a canonical way, and $f$ is discrete holomorphic. Each of the eight vertices of the surface of genus three is a branch point of multiplicity two.
\end{example}

\begin{remark}
Note that even if $f$ is not globally biconstant, it may have vanishing points. The reason for saying that quadrilaterals where $f$ is biconstant are branch points of multiplicity two is that if we go along the vertices $b_-,w_-,b_+,w_+$ of $Q$, then its images are $f(b_-),f(w_-),f(b_-),f(w_-)$ (Figure~\ref{fig:merging}). However, in combination with vanishing points, this definition of branching might be misleading. It is more appropriate to consider a finite subgraph $\Diamond_0\subseteq\Diamond$ that forms a simply-connected closed region consisting of $F$ quadrilaterals, $I$ interior points $V_{\textnormal{int}}$ (all of them vanishing points), and $B=2(F-I+1)$ boundary points (all of them regular points) as one single branch point of multiplicity $F-I+1$. Indeed, black and white points alternate at the boundary and they are always mapped to the same black or white image point, respectively. In terms of branch numbers this interpretation is fine since \[F-I=\sum_{Q \in V(\Diamond_0)} b_f(Q) + \sum_{v \in V_{\textnormal{int}}} b_f(v).\]

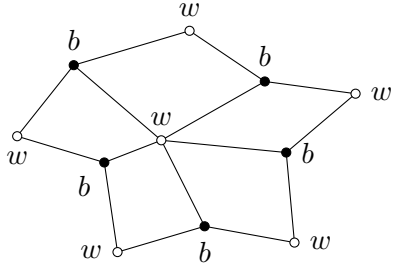
\begin{figure}[htbp]
		\centering
		\beginpgfgraphicnamed{merging}
		\begin{tikzpicture}
		[white/.style={circle,draw=black,fill=white,thin,inner sep=0pt,minimum size=1.2mm},
black/.style={circle,draw=black,fill=black,thin,inner sep=0pt,minimum size=1.2mm},
gray/.style={circle,draw=black,fill=gray,thin,inner sep=0pt,minimum size=1.2mm},scale=0.6]
	\clip(1.8,-4) rectangle (10.8,2);
				\draw (2.4,-1.03)-- (3.64,0.54);
				\draw (3.64,0.54)-- (6.18,1.3);
				\draw (3.64,0.54)-- (5.56,-1.12);
				\draw (5.56,-1.12)-- (4.31,-1.61);
				\draw (4.31,-1.61)-- (2.4,-1.03);
				\draw (4.31,-1.61)-- (4.6,-3.59);
				\draw (9.82,-0.09)-- (7.83,0.18);
				\draw (9.82,-0.09)-- (8.3,-1.39);
				\draw (7.83,0.18)-- (5.56,-1.12);
				\draw (8.3,-1.39)-- (5.56,-1.12);
				\draw (7.83,0.18)-- (6.18,1.3);
				\draw (5.56,-1.12)-- (6.51,-3.02);
				\draw (6.51,-3.02)-- (4.6,-3.59);
				\draw (6.51,-3.02)-- (8.48,-3.38);
				\draw (8.48,-3.38)-- (8.3,-1.39);
					\node[white] (w1) [label=right:$w$] at (9.82,-0.09) {};
					\node[white] (w2) [label=above:$w$] at (5.56,-1.12) {};
					\node[white] (w3) [label=below:$w$] at (2.4,-1.03) {};
					\node[white] (w4) [label=right:$w$] at (8.48,-3.38) {};
					\node[white] (w5) [label=left:$w$] at (4.6,-3.59) {};
					\node[white] (w6) [label=above:$w$] at (6.18,1.3) {};
					\node[black] (b1) [label=above:$b$] at (7.83,0.18) {};
					\node[black] (b2) [label=right:$b$] at (8.3,-1.39) {};
					\node[black] (b3) [label=above:$b$] at (3.64,0.54) {};
					\node[black] (b4) [label=below:$b$] at (6.51,-3.02) {};
					\node[black] (b5) [label=below left:$b$] at (4.31,-1.61) {};
			\end{tikzpicture}
		\endpgfgraphicnamed
   \caption{A branch point of multiplicity 5=5-1+1 (labels indicate image points)}
   \label{fig:merging}
\end{figure}

\end{remark}

\begin{corollary}\label{cor:surjective}
Let $f:V(\Lambda) \to V(\Lambda')$ be discrete holomorphic and not biconstant. Then, $f$ is surjective. If in addition $\Sigma$ is compact, then $\Sigma'$ is compact as well.
\end{corollary}
\begin{proof}
Assume that $f$ is not surjective. Then, there is $v'\in V(\Lambda')$ not contained in the image. Say $v'$ is black. Take $v_0' \in f(\Gamma)$ combinatorially closest to $v'$. Since all black neighbors of a black vanishing point of $f$ have the same image and $f$ is not biconstant, there is a regular point $v_0$ in the preimage of $v_0'$. By Lemma~\ref{lem:open_map}, the image of the star of $v_0$ equals the star of $v_0'$. Thus, there is an image point combinatorially nearer to $v'$ as $v_0'$, contradiction.

If $\Sigma$ is compact, then $\Lambda$ is finite. So $\Lambda'$ is finite as well and $\Sigma'$ is compact.
\end{proof}

\begin{corollary}\label{cor:Liouville}
Let $\Sigma$ be compact and $\Sigma'$ be homeomorphic to a plane. Then, any discrete holomor\-phic mapping $f:V(\Lambda) \to V(\Lambda')$ is biconstant.
\end{corollary}

Note that we will prove the more general discretization of Liouville's theorem that any complex valued discrete holomorphic function $f:V(\Lambda) \to \mC$ on a compact discrete Riemann surface is biconstant later in Theorem~\ref{th:Liouville}.

\begin{theorem}\label{th:degree}
Let $f:V(\Lambda) \to V(\Lambda')$ be a discrete holomorphic mapping. Then, there exists a number $N \in \mZ_{\geq 0} \cup \left\{ \infty \right\}$ such that for all $v' \in V(\Lambda')$: \[N=\sum_{v\in f^{-1}(v')}\left(b_f(v)+1\right).\] Furthermore, for any $Q' \in F(\Lambda')$, $N$ equals the number of $Q \in F(\Lambda)$ such that $f$ maps the vertices of $Q$ bijectively to the vertices of $Q'$.
\end{theorem}

\begin{proof}
If $f$ is biconstant, then all $b_f(v)+1$ are zero and $N=0$ fulfills the requirements.

Assume now that $f$ is not biconstant. By Corollary~\ref{cor:surjective}, $f$ is surjective. Let $Q' \in F(\Lambda')$ and let $v'$ be a vertex of $Q'$. We want to count the number $N$ of $Q \in F(\Lambda)$ such that $f$ maps the vertices of $Q$ bijectively to the vertices of $Q'$. Let $v \in f^{-1}(v')$. By Lemma~\ref{lem:open_map}, exactly $b_f(v)+1$ quadrilaterals incident to $v$ are mapped bijectively to $Q'$. Conversely, any $Q \in F(\Lambda)$ such that $f$ maps the vertices of $Q$ bijectively to the vertices of $Q'$ has exactly one vertex in the preimage of $f^{-1}(v')$. Therefore, \[N=\sum_{v\in f^{-1}(v')}\left(b_f(v)+1\right).\] The same formula holds true if we replace $Q'$ by another face incident to $v'$ or $v'$ by some other vertex incident to $Q'$. Thus, $N$ does not depend on the choice of the face $Q'$ and the incident vertex $v'$.
\end{proof}

\begin{definition}
If $N>0$, then $f$ is called an \textit{$N$-sheeted discrete holomorphic covering}.
\end{definition}

\begin{remark}
If $\Sigma$ is compact, then $N<\infty$. The characterization of $N$ as the number of preimage quadrilaterals nicely explains why $N$ is called the number of sheets of $f$. However, a quadrilateral of $\Lambda$ corresponds to one of the $N$ sheets (and not to just two single points) only if $f$ is not biconstant there.
\end{remark}

Finally, we state and prove a \textit{discrete Riemann-Hurwitz formula}.

\begin{theorem}\label{th:Riemann_Hurwitz}
Let $\Sigma$ be compact and $f:V(\Lambda) \to V(\Lambda')$ be an $N$-sheeted discrete holomorphic covering of the compact discrete Riemann surface $\Sigma'$ of genus $g'$. Then, the genus $g$ of $\Sigma$ is equal to \[g=N(g'-1)+1+\frac{b}{2},\] where $b$ is the total branching number of $f$: \[b=\sum_{v \in V(\Lambda)} b_f(v) + \sum_{Q \in V(\Diamond)} b_f(Q).\]
\end{theorem}

\begin{proof}
Since we consider quad-decompositions, the number of edges of $\Lambda$ equals twice the number of faces. Thus, the Euler characteristic $2-2g$ of $\Sigma$ is given by $|V(\Lambda)|-|V(\Diamond)|$. By Theorem~\ref{th:degree}, \[|V(\Lambda)|=N|V(\Lambda')|-\sum_{v \in V(\Lambda)} b_f(v).\] If we count the number of faces of $\Lambda$, then we have $N|V(\Diamond')|$ quadrilaterals that are mapped to a complete quadrilateral of $\Lambda'$ by Theorem~\ref{th:degree} and $\sum_{Q \in V(\Diamond)} b_f(Q)$ faces are mapped to an edge of $\Lambda'$. Hence, \[|V(\Diamond)|=N|V(\Diamond')|+\sum_{Q \in V(\Diamond)} b_f(Q).\]

\begin{align*}2-2g=|V(\Lambda)|-|V(\Diamond)|&=N|V(\Lambda')|-\sum_{v \in V(\Lambda)} b_f(v)-N|V(\Diamond')|-\sum_{Q \in V(\Diamond)} b_f(Q)=N(2-2g')-b\end{align*} now implies the final result.
\end{proof}

\begin{example}
In the example depicted in Figure~\ref{fig:cover}, $g=3$, $g'=0$, $N=2$, and $b=8$. \[3=2\cdot(0-1)+1+\frac{8}{2}\] then demonstrates the validity of the discrete Riemann-Hurwitz formula. 
\end{example}


\section{Discrete exterior calculus}\label{sec:differentials}

In this section, we consider a discrete Riemann surface $(\Sigma, \Lambda, z)$ and adapt the fundamental notions and properties of discrete complex analysis discussed in \cite{BoG15} to discrete Riemann surfaces. All omitted proofs can be literally translated from \cite{BoG15} to the more general setting of discrete Riemann surfaces.

Note that our treatment of discrete exterior calculus is similar to Mercat's approach in \cite{Me01,Me07,Me08}. However, in Section~\ref{sec:differential_forms}
we suggest a different notation of multiplication of functions with discrete one-forms, leading to a discrete exterior derivative that is defined on a larger class of discrete one-forms in Section~\ref{sec:derivative}. It coincides with Mercat's discrete exterior derivative in the case of discrete one-forms of type $\Diamond$ that he considers. In contrast, our definitions mimic the coordinate representation of the smooth theory. Still, our definitions of a discrete wedge product in Section~\ref{sec:wedge} and a discrete Hodge star in Section~\ref{sec:hodge} are equivalent to Mercat's in \cite{Me08}.


\subsection{Discrete differential forms} \label{sec:differential_forms}

The most important type of functions is $f:V(\Lambda)\to\mC$, but in local charts complex functions defined on subsets of $V(\Diamond)$ such as $\partial_\Lambda f$ occur as well.

\begin{definition}
A \textit{discrete one-form} or \textit{discrete differential} $\omega$ is a complex function on the oriented edges of the medial graph $X_0$ such that $\omega(-e)=\omega(e)$ for any oriented edge $e$ of $X_0$. Here, $-e$ denotes the edge $e$ with opposite orientation.

The evaluation of $\omega$ at an oriented edge $e$ of $X_0$ is denoted by $\int_e \omega$. For a directed path $P$ in $X_0$ consisting of oriented edges $e_1,e_2,\ldots,e_n$, the \textit{discrete integral} along $P$ is defined as $\int_P \omega=\sum_{k=1}^n \int_{e_k} \omega$. For closed paths $P$, we write $\oint_P \omega$ instead.
\end{definition}
\begin{remark}
If we speak about discrete one-forms or discrete differentials and do not specify their domain, then we will always assume that they are defined on oriented edges of the whole medial graph $X$.
\end{remark}

Of particular interest are discrete one-forms that actually come from discrete one-forms on $\Gamma$ and $\Gamma^*$.

\begin{definition}
A discrete one-form $\omega$ defined on the oriented edges of $X_0$ is of \textit{type} $\Diamond$ if for any quadrilateral $Q \in V(\Diamond_0)$ and its incident black (or white) vertices $v,v'$ the equality $\omega([Q,v])=-\omega([Q,v'])$ holds. The latter two edges inherit their orientation from $\partial F_Q$.
\end{definition}

\begin{definition}
A \textit{discrete two-form} $\Omega$ is a complex function on $F(X_0)$.

The evaluation of $\Omega$ at a face $F$ of $X_0$ is denoted by $\iint_F \Omega$. If $S$ is a set of faces $F_1,\ldots, F_n$ of $X_0$, then $\iint_S \Omega=\sum_{k=1}^n \iint_{F_k} \Omega$ defines the \textit{discrete integral} of $\Omega$ over $S$.

$\Omega$ is of \textit{type} $\Lambda$ if $\Omega$ vanishes on all faces of $X_0$ corresponding to $V(\Diamond_0)$ and of \textit{type} $\Diamond$ if $\Omega$ vanishes on all faces of $X_0$ corresponding to $V(\Lambda_0)$.
\end{definition}
\begin{remark}
Discrete two-forms of type $\Lambda$ or type $\Diamond$ correspond to functions on $V(\Lambda_0)$ or $V(\Diamond_0)$ by the discrete Hodge star that will be defined later in Section~\ref{sec:hodge}.
\end{remark}

\begin{definition}
Let for short $z$ be a chart $z_Q$ of a quadrilateral $Q \in V(\Diamond)$ or a chart $z_v$ of the star of a vertex $v \in V(\Lambda)$. On its domain, the discrete one-forms $dz$ and $d\bar{z}$ are defined in such a way that $\int_e dz=z(e)$ and $\int_e d\bar{z}=\overline{z(e)}$ hold for any oriented edge $e$ of $X$. The discrete two-forms $\Omega_\Lambda^z$ and $\Omega_\Diamond^z$ are zero on faces of $X$ corresponding to vertices of $\Diamond$ or $\Lambda$, respectively, and defined by \[\iint_F \Omega_\Lambda^z=-4i\textnormal{area}(z(F)) \textnormal{ and } \iint_F \Omega_\Diamond^z=-4i\textnormal{area}(z(F))\] on faces $F$ corresponding to vertices of $\Lambda$ or $\Diamond$, respectively. Here, $\textnormal{area}(z(F))$ denotes the algebraic area of the polygon $z_v(F_v)$ or the Euclidean area of the parallelogram $z(F)$, respectively.
\end{definition}

\begin{remark}
Our main objects either live on the quad-decomposition $\Lambda$ or on its dual $\Diamond$. Thus, we have to deal with two different cellular decompositions at the same time. The medial graph has the crucial property that its faces split into two sets which are respectively $\Lambda=\Diamond^*$ and $\Diamond=\Lambda^*$. Furthermore, the Euclidean area of the Varignon parallelogram inside a quadrilateral $z(Q)$ is just half of its area. In an abstract sense, a corresponding statement is true for the cells of $X$ corresponding to vertices of $\Lambda$ and the faces of $\Diamond$. This statement can be made precise in the setting of planar parallelogram-graphs, see \cite{BoG15}. For this reason, the additional factor of two is necessary to make $\Omega_\Lambda^z$ and $\Omega_\Diamond^z$ the straightforward discretizations of $dz \wedge d\bar{z}$. As it turns out in Section~\ref{sec:wedge}, $\Omega_\Diamond^z$ is indeed the discrete wedge product of $dz$ and $d\bar{z}$.
\end{remark}

\begin{definition}
Let $f:V(\Lambda_0)\to\mC$, $h:V(\Diamond_0)\to\mC$, $\omega$ a discrete one-form defined on the oriented edges of $X_0$, and $\Omega_1,\Omega_2$ discrete two-forms defined on $F(X_0)$ that are of type $\Lambda$ and $\Diamond$, respectively. For any oriented edge $e=[Q,v]$ and any faces $F_v, F_Q$ of $X_0$ corresponding to $v\in V(\Lambda_0)$ or $Q \in V(\Diamond_0)$, we define the products $f\omega$, $h\omega$, $f\Omega_1$, and $h\Omega_2$ by
\begin{align*}
\int_{e}f\omega:&=f(v)\int_{e}\omega \ \quad \textnormal{ and } \quad \iint_{F_v} f\Omega_1:=f(v)\iint_{F_v}\Omega_1, \quad \iint_{F_Q} f\Omega_1:=0;\\
\int_{e}h\omega:&=h(Q)\int_{e}\omega \quad \textnormal{ and } \quad \iint_{F_v} h\Omega_2:=0, \qquad \qquad \qquad \; \! \iint_{F_Q} h\Omega_2:=h(Q)\iint_{F_Q}\Omega_2.
\end{align*}
\end{definition}

\begin{remark}
A discrete one-form of type $\Diamond$ can be locally represented as $pdz_Q+qd\bar{z}_Q$ on all edges of a face of $X$ corresponding to $Q \in V(\Diamond)$, where $p,q \in \mC$. Similarly, we could define discrete one-forms of type $\Lambda$. However, this notion would depend on the chart and would be not well-defined on a discrete Riemann surface.
\end{remark}


\subsection{Discrete derivatives and Stokes' theorem}\label{sec:derivative}

\begin{definition}
Let $Q \in V(\Diamond)\cong F(\Lambda)$ and $f$ be a complex function on the vertices of $Q$. In addition, let $v\in V(\Lambda)$ and $h$ be a complex function defined on all quadrilaterals $Q_s \sim v$. Let $F_Q$ and $F_v$ be the faces of $X$ corresponding to $Q$ and $v$ with counterclockwise orientations of their boundaries. Then, the \textit{discrete derivatives} $\partial_\Lambda f$, $\bar{\partial}_\Lambda f$ in the chart $z_Q$ and $\partial_\Diamond h$, $\bar{\partial}_\Diamond h$ in the chart $z_v$ are defined by
\begin{align*}
 \partial_\Lambda f(Q)&:=\frac{1}{\iint_{F_Q} \Omega_\Diamond^{z_Q}}\oint_{\partial F_Q} f d\bar{z}_Q, \qquad \bar{\partial}_\Lambda f (Q):=\frac{-1}{\iint_{F_Q} \Omega_\Diamond^{z_Q}}\oint_{\partial F_Q} f dz_Q;\\
 \partial_\Diamond h(v)&:=\frac{1}{\iint_{F_v} \Omega_\Lambda^{z_v}}\oint_{\partial F_v} h d\bar{z}_v, \qquad \quad \bar{\partial}_\Diamond h(v):=\frac{-1}{\iint_{F_v} \Omega_\Lambda^{z_v}}\oint_{\partial F_v} h dz_v.
\end{align*}

$h$ is said to be \textit{discrete holomorphic} in the chart $z_v$ if $\bar{\partial}_\Diamond h (v)=0$.
\end{definition}

As in the classical theory, the discrete derivatives depend on the chosen chart. We do not include these dependences in the notions, but it will be clear from the context which chart is used.

\begin{remark}
Whereas discrete holomorphicity for functions $f:V(\Lambda) \to \mC$ is well-defined and equivalent to $\bar{\partial}_\Lambda f (Q)=0$ in any chart $z_Q$ (see \cite{BoG15}), discrete holomorphicity of functions on $V(\Diamond)$ is not consistently defined by the discrete complex structure. Indeed, if $\rho_Q=1$ for all faces $Q$ incident to $v \in V(\Lambda)$, then any cyclic polygon with the correct number of vertices can be the image of the vertices adjacent to $v$ under a chart $z_v$ compatible with the discrete complex structure, but the equation $\bar{\partial}_\Diamond h (v)=0$ depends on the choice of the cyclic polygon. 
\end{remark}

\begin{definition}
Let $f:V(\Lambda_0) \to \mC$ and $h:V(\Diamond_0) \to \mC$. We define the \textit{discrete exterior derivatives} $df$ and $dh$ on the edges of $X_0$ in a chart $z$ as follows:
\begin{align*}
 df:=\partial_\Lambda f dz+\bar{\partial}_\Lambda f d\bar{z}, \quad dh:=\partial_\Diamond h dz+\bar{\partial}_\Diamond h d\bar{z}.
\end{align*}

Let $\omega$ be a discrete one-form defined on all boundary edges of a face $F_v$ of the medial graph $X$ corresponding to $v\in V(\Lambda)$ or on all four boundary edges of a face $F_Q$ of $X$ corresponding to $Q\in F(\Lambda)$. In a chart $z$ around $F_v$ or $F_Q$, respectively, we write $\omega=p dz+ q d\bar{z}$ with functions $p,q$ defined on faces incident to $v$ or vertices incident to $Q$, respectively. The \textit{discrete exterior derivative} $d\omega$ is given by
\begin{align*}
 d\omega|_{F_v}&:=\left(\partial_\Diamond q - \bar{\partial}_\Diamond p\right)  \Omega_\Lambda^z,\\
 d\omega|_{F_Q}&:=\left(\partial_\Lambda q - \bar{\partial}_\Lambda p\right) \Omega_\Diamond^z.
\end{align*}
\end{definition}

The representation of $\omega$ as $p dz+ q d\bar{z}$ ($p,q$ defined on edges of $X$) we have used above may be nonunique. However, $d\omega$ is well-defined and does not depend on the chosen chart by \textit{discrete Stokes' theorem}.

\begin{theorem}\label{th:stokes}
Let $f:V(\Lambda_0) \to \mC$ and $\omega$ be a discrete one-form defined on oriented edges of $X_0$. Then, for any directed edge $e$ of $X_0$ starting in the midpoint of the edge $vv'_-$ and ending in the midpoint of the edge $vv'_+$ of $\Lambda_0$ and for any finite collection of faces $F$ of $X_0$ with counterclockwise oriented boundary $\partial F$ we have:

\begin{align*}
 \int_e df&=\frac{f(v'_+)-f(v'_-)}{2}=\frac{f(v)+f(v'_+)}{2}-\frac{f(v)+f(v'_-)}{2};\\
 \iint_F d\omega&=\oint_{\partial F} \omega.
\end{align*}
\end{theorem}

\begin{definition}
Let $\Diamond_0 \subseteq \Diamond$ form a simply-connected closed region. A discrete one-form $\omega$ defined on oriented edges of $X_0$ is said to be \textit{closed} if $d\omega\equiv 0$.
\end{definition}

\begin{proposition}\label{prop:dd0}
 Let $f:V(\Lambda) \to \mC$. Then, $ddf=0$.
\end{proposition}

\begin{corollary}\label{cor:commutativity}
Let $f$ be a function defined on the vertices of all quadrilaterals incident to $v \in V(\Lambda)$. Then, $\partial_\Diamond\bar{\partial}_\Lambda f(v)=\bar{\partial}_\Diamond\partial_\Lambda f(v)$ in a chart $z_v$ of the star of $v$. In particular, $\partial_\Lambda f$ is discrete holomorphic in $z_v$ if $f$ is discrete holomorphic.
\end{corollary}

\begin{corollary}\label{cor:f_holomorphic}
 Let $f:V(\Lambda) \to \mC$. Then, $f$ is discrete holomorphic at all faces incident to $v\in V(\Lambda)$ if and only if in a chart $z_v$ around $v$, $df=p dz_v$ for some function $p$ defined on the faces incident to $v$. In this case, $p$ is discrete holomorphic in $z_v$.
\end{corollary}

\begin{definition}
A discrete differential $\omega$ of type $\Diamond$ is \textit{discrete holomorphic} if $d\omega=0$ and if in any chart $z_Q$ of a quadrilateral $Q\in V(\Diamond)$, $\omega=p dz_Q$. $\omega$ is \textit{discrete antiholomorphic} if $\bar{\omega}$ is discrete holomorphic. 
\end{definition}
\begin{remark}
It suffices to check this condition for just one chart of $Q$, as follows from Lemma~\ref{lem:Hodge_projection} below. In particular, discrete holomorphicity of discrete one-forms depends on the discrete complex structure only. If $\omega$ is discrete holomorphic, then we can write $\omega=p dz_v$ in a chart $z_v$ around $v \in V(\Lambda)$, where $p$ is a function defined on the faces incident to $v$. In this case, $p$ is discrete holomorphic in $z_v$. Conversely, the closeness condition can be replaced by requiring that $p$ is discrete holomorphic.
\end{remark}

\begin{proposition}\label{prop:primitive2}
Let $\Diamond_0 \subseteq \Diamond$ form a simply-connected closed region and let $\omega$ be a closed discrete differential of type $\Diamond$ defined on oriented edges of $X_0$. Then, there is a function $f:=\int \omega :V(\Lambda_0)\to\mC$ such that $\omega=df$. $f$ is unique up to two additive constants on $\Gamma_0$ and $\Gamma^*_0$. If $\omega$ is discrete holomorphic, then $f$ is as well.
\end{proposition}


\subsection{Discrete wedge product}\label{sec:wedge}

Let $\omega$ be a discrete one-form of type $\Diamond$. Then, for any chart $z_Q$ of a quadrilateral $Q\in V(\Diamond)$ there is a unique representation $\omega|_{\partial F_Q}=p dz_Q+q d\bar{z}_Q$ with complex numbers $p$ and $q$. To calculate them, one can first construct a function $f$ on the vertices of $Q$ such that $\omega|_{\partial F_Q}=df$ and then take $p=\partial_\Lambda f$ and $q=\bar{\partial}_\Lambda f$, see \cite{BoG15}.

\begin{definition}
Let $\omega,\omega'$ be two discrete one-forms of type $\Diamond$ defined on the oriented edges of $X_0$. Then, the \textit{discrete wedge product} $\omega\wedge\omega'$ is defined as the discrete two-form of type $\Diamond$ that equals \[\left(pq'-qp'\right)\Omega_\Diamond^{z_Q}\] on a face $F_Q$ corresponding to $Q\in V(\Diamond)$. Here, $z_Q$ is a chart of $Q$ and $\omega|_{\partial F_Q}=p dz_Q+ q d\bar{z}_Q$ and $\omega'|_{\partial F_Q}=p' dz_Q+ q' d\bar{z}_Q$.
\end{definition}

The following proposition connects our definition of a discrete wedge product with Mercat's in \cite{Me01,Me07,Me08} and also shows that the discrete wedge product does not depend on the choice of the chart.

\begin{proposition}\label{prop:wedge_Mercat}
Let $F_Q$ be the face of $X$ corresponding to $Q\in V(\Diamond)$, let $z_Q$ be a chart, and let $e,e^*$ be the oriented edges of $X$ parallel to the black and white diagonal of $Q$, respectively, such that $\im \left(z_Q(e^*)/z_Q(e)\right)>0$. Then, \[\iint_{F_Q} \omega\wedge\omega' =  2\int_e \omega \int_{e^*} \omega'- 2\int_{e^*} \omega \int_e \omega'.\]
\end{proposition}

Finally, the discrete exterior derivative is a derivation for the wedge product:

\begin{theorem}\label{th:derivation}
Let $f:V(\Lambda_0) \to \mC$ and $\omega$ be a discrete one-form of type $\Diamond$ defined on the oriented edges of $X_0$. Then, the following identity holds on $F(X_0)$: \[d(f\omega)=df\wedge\omega+fd\omega.\]
\end{theorem}


\subsection{Discrete Hodge star and discrete Laplacian}\label{sec:hodge}

\begin{definition}
Let $\Omega_{\Sigma}$ be a fixed nowhere vanishing discrete two-form, $f:F(\Lambda_0)\to\mC$, $h:V(\Diamond_0)\to\mC$, $\omega$ a discrete one-form of type $\Diamond$ defined on oriented edges of $X_0$, and $\Omega$ a discrete two-form either of type $\Lambda$ or $\Diamond$. In a chart $z_Q$ of $Q\in V(\Diamond)$, we write $\omega|_{\partial F_Q}=p dz_Q +qd\bar{z}_Q$. Then, the \textit{discrete Hodge star} is defined by \begin{align*} \star f:= f \Omega_\Sigma; \quad \star h:= h \Omega_\Sigma; \quad \star \omega|_{\partial F_Q}:=-ip dz_Q+iq d\bar{z}_Q;\quad \star \Omega&:=\frac{\Omega}{\Omega_\Sigma}.\end{align*}
\end{definition}

\begin{remark}
In the planar case, the choice of $\Omega_{\Sigma}=i/2 \Omega_\Diamond^z$ on faces of $X$ corresponding to faces of the quad-graph and $\Omega_{\Sigma}=i/2 \Omega_\Lambda^z$ on faces corresponding to vertices is the most natural one. Throughout the remainder of this chapter, $\Omega_{\Sigma}$ is a fixed positive real two-form on $(\Sigma,\Lambda,z)$.

In the classical setup, there is a canonical nonvanishing two-form coming from a complete Riemannian metric of constant curvature. An interesting question is whether there exists some canonical two-form for discrete Riemann surfaces as we defined them. Note that the nonlinear theory developed in \cite{BoPSp10} contains a uniformization of discrete Riemann surfaces and discrete metrics with constant curvature.
\end{remark}

\begin{proposition}\label{prop:hodge_Mercat}
Let $Q\in V(\Diamond)$ with chart $z_Q$, and let $e,e^*$ be oriented edges of $X$ parallel to the black and white diagonal of $Q$, respectively, such that $\im \left(e^*/e\right)>0$. If $\omega$ is a discrete one-form of type $\Diamond$ defined on the oriented edges of the boundary of the face of $X$ corresponding to $Q$, then
 \begin{align*}
 \int_e \star\omega&=\cot\left(\varphi_Q\right) \int_e \omega-\frac{|z_Q(e)|}{|z_Q(e^*)| \sin\left(\varphi_Q\right)}\int_{e^*}\omega,\\
 \int_{e^*} \star\omega&=\frac{|z_Q(e^*)|}{|z_Q(e)| \sin\left(\varphi_Q\right)} \int_e \omega-\cot\left(\varphi_Q\right)\int_{e^*}\omega.
\end{align*}
\end{proposition}

Proposition~\ref{prop:hodge_Mercat} shows not only that our definition of a discrete Hodge star on discrete one-forms does not depend on the chosen chart, but also that it coincides with Mercat's definition given in \cite{Me08}.

Clearly, $\star^2=-\textnormal{Id}$ on discrete differentials of type $\Diamond$ and $\star^2=\textnormal{Id}$ on complex functions and discrete two-forms. The next lemma shows that discrete holomorphic differentials are well-defined.

\begin{lemma}\label{lem:Hodge_projection}
Let $Q \in V(\Diamond)$ and $F_Q$ be the face of $X$ corresponding to $Q$. A discrete differential $\omega$ of type $\Diamond$ defined on the oriented edges of $F_Q$ is of the form $\omega=p dz_Q$ (or $\omega=q d\bar{z}_Q$) in any chart $z_Q$ of $Q\in V(\Diamond)$ if and only if $\star\omega=-i\omega$ (or $\star\omega=i\omega$).
\end{lemma}
\begin{proof}
Let us take a (unique) representation $\omega=p dz_Q+ q d\bar{z}_Q$ in a coordinate chart $z_Q$ of $Q\in V(\Diamond)$. By definition, $\star\omega=-i\omega$ is equivalent to $q=0$. Analogously, $\star\omega=i\omega$ is equivalent to $p=0$.
\end{proof}

\begin{definition}
If $\omega$ and $\omega'$ are both discrete differentials of type $\Diamond$ defined on oriented edges of $X$, then we define their \textit{discrete scalar product} \[\langle \omega, \omega' \rangle:=\iint_{F(X)} \omega \wedge \star\bar{\omega}',\] whenever the right hand side converges absolutely. In a similar way, the discrete scalar product between two discrete two-forms or two complex functions on $V(\Lambda)$ is defined.
\end{definition}

A calculation in local coordinates shows that $\langle \cdot,\cdot \rangle$ is indeed a Hermitian scalar product.

\begin{definition}
$L_2(\Sigma,\Lambda,z)$ is the Hilbert space of \textit{square integrable} discrete differentials with respect to $\langle \cdot,\cdot \rangle$.
\end{definition}

\begin{proposition}\label{prop:adjoint2}
$\delta:=-\star d \star$ is the \textit{formal adjoint} of the discrete exterior derivative $d$:

Let $f:V(\Lambda)\to\mC$, let $\omega$ be a discrete one-form of type $\Diamond$, and let $\Omega:F(X)\to\mC$ be a discrete two-form of type $\Lambda$. If all of them are compactly supported, then \[\langle df, \omega \rangle =\langle f, \delta \omega \rangle \textnormal{ and }\langle d\omega,\Omega\rangle= \langle \omega, \delta \Omega\rangle.\]
\end{proposition}

\begin{definition}
The \textit{discrete Laplacian} on functions $f:V(\Lambda)\to\mC$, discrete one-forms of type $\Diamond$, or discrete two-forms on $F(X)$ of type $\Lambda$ is defined as the linear operator \[\triangle:=-\delta d-d\delta=\star d \star d +d \star d \star.\]
$f:V(\Lambda)\to\mC$ is said to be \textit{discrete harmonic} at $v\in V(\Lambda)$ if $\triangle f(v)=0$. 
\end{definition}

\begin{remark}
Note that straight from the definition and Corollary~\ref{cor:commutativity}, it follows for $f:V(\Lambda)\to\mC$ that $\triangle f (v)$ is proportional to $4\partial_\Diamond\bar{\partial}_\Lambda f(v)=4\bar{\partial}_\Diamond\partial_\Lambda f(v)$ in the chart $z_v$ around $v \in V(\Lambda)$. In particular, discrete harmonicity of functions does not depend on the choice of $\Omega_\Sigma$, and discrete holomorphic functions are discrete harmonic.
\end{remark}

\begin{lemma}\label{lem:Dirichlet_boundary2}
Let $(\Sigma,\Lambda,z)$ be a compact discrete Riemann surface. Then, the discrete Dirichlet energy functional $E_\Diamond$ defined by $E_{\Diamond}(f):=\langle df,df \rangle$ for functions $f:V(\Lambda) \to \mR$ is a convex nonnegative quadratic functional in the vector space of real functions on $V(\Lambda)$. Furthermore, \[-\frac{\partial E_{\Diamond}}{\partial f(v)}(f)=2\triangle f(v)\iint_{F_v} \Omega_{\Sigma} \] for any $v \in V(\Lambda)$. In particular, extremal points of this functional are functions that are discrete harmonic everywhere.
\end{lemma}

As a conclusion of this section, we state and prove \textit{discrete Liouville's theorem}.

\begin{theorem}\label{th:Liouville}
Let $(\Sigma,\Lambda,z)$ be a compact discrete Riemann surface. Then, any discrete harmonic function $f:V(\Lambda)\to\mC$ is biconstant. In particular, any complex valued discrete holomorphic function is biconstant.
\end{theorem}
\begin{proof}
Since $\delta$ is the formal adjoint of $d$ by Proposition~\ref{prop:adjoint2},
\[\langle df,df \rangle=\langle f, \delta d f \rangle=\langle f, \triangle f \rangle=0.\] Now, $\langle df,df \rangle \geq 0$ and equality holds only if $df=0$, i.e., if $f$ is biconstant.
\end{proof}


\section{Periods of discrete differentials}\label{sec:periods}

In this section, we define the (discrete) periods of a closed discrete differential of type $\Diamond$ on a compact discrete Riemann surface $(\Sigma, \Lambda, z)$ of genus $g$ in Section~\ref{sec:cover} and state and prove a discrete Riemann bilinear identity in Section~\ref{sec:RBI}. Although we aim at being as close as possible to the smooth case in our presentation, the bipartite structure of $\Lambda$ prevents us from doing so. We struggle with the same problem of white and black periods as Mercat did for discrete Riemann surfaces whose discrete complex structure is described by real numbers $\rho_Q$ in \cite{Me07}. The reason for this is that a discrete differential of type $\Diamond$ corresponds to a pair of discrete differentials on each of $\Gamma$ and $\Gamma^*$.

Mercat constructed out of a canonical homology basis on $\Lambda$ certain canonical homology bases on $\Gamma$ and $\Gamma^*$. By solving a discrete Neumann problem, he then proved the existence of dual cohomology bases on $\Gamma$ and $\Gamma^*$. The discrete Riemann bilinear identity for the elements of the bases (and by linearity for general closed discrete differentials) was a direct consequence of the construction.

On the contrary, the proof given in \cite{BoSk12} followed the ideas of the smooth case, but the relation to discrete wedge products was not that immediate. We will give a full proof of the general discrete Riemann bilinear identity that follows the lines of the proof of the classical Riemann bilinear identity, using almost the same notation. The main difference to \cite{BoSk12} is that we use a different refinement of the cellular decomposition to profit of a cellular decomposition of the canonical polygon with $4g$ vertices. The appearance of black and white periods indicates the analogy to Mercat's approach in \cite{Me07}.


\subsection{Universal cover and periods}\label{sec:cover}

Let $p: \tilde{\Sigma} \to \Sigma$ denote the universal covering of the compact surface $\Sigma$. $p$ gives rise to a bipartite quad-decomposition $\tilde{\Lambda}$ with medial graph $\tilde{X}$ and a covering $p: \tilde{\Lambda} \to \Lambda$. Now, $(\tilde{\Sigma}, \tilde{\Lambda}, z \circ p)$ is a discrete Riemann surface as well and $p: V(\tilde{\Lambda})\to V(\Lambda)$ is a discrete holomorphic mapping.

We fix a base vertex $\tilde{v}_0 \in V(\tilde{\Lambda})$. Let $\alpha_1, \ldots, \alpha_g, \beta_1, \ldots, \beta_g$ be smooth loops on $\Sigma$ with base point $v_0:=p(\tilde{v}_0)$ such that these loops cut out a fundamental $4g$-gon $F_g$. It is well known that such loops exist; the order of loops at the boundary of $F_g$ is $\alpha_k, \beta_k, \alpha_k^{-1}, \beta_k^{-1}$, $k$ going in order from $1$ to $g$. Their homology classes $a_1, \ldots, a_g, b_1, \ldots, b_g$ form a canonical homology basis of $H_1(\Sigma,\mZ)$.

Clearly, there are homotopies between $\alpha_1, \ldots, \alpha_g, \beta_1, \ldots, \beta_g$ and closed paths $\alpha'_1, \ldots, \alpha'_g, \beta'_1, \ldots, \beta'_g$ on $X$, all of the latter having the same fixed base point $x_0\in V(X)$.

\begin{definition}
Let $P$ be an oriented cycle on $X$. $P$ induces closed paths on $\Gamma$ and $\Gamma^*$ that we denote by $B(P)$ and $W(P)$ in the following way: For an oriented edge $[Q,v]$ of $P$, we add the black (or white) vertex $v$ to $B(P)$ (or $W(P)$) and the corresponding white (or black) diagonal of $Q \in F(\Lambda)$ to $W(P)$ (or $B(P)$), see Figure~\ref{fig:contours2}. The orientation of the diagonal is induced by the orientation of $[Q,v]$. Clearly, $B(P)$ and $W(P)$ are cycles on $\Gamma$ and $\Gamma^*$ that are homotopic to $P$. We denote the one-chains on $X$ consisting of all the black or white edges corresponding to $B(P)$ and $W(P)$ by $BP$ and $WP$, respectively.
\end{definition}

\begin{figure}[htbp]
\begin{center}
\beginpgfgraphicnamed{medial}
\begin{tikzpicture}
[white/.style={circle,draw=black,fill=black,thin,inner sep=0pt,minimum size=1.2mm},
black/.style={circle,draw=black,fill=white,thin,inner sep=0pt,minimum size=1.2mm},
gray/.style={circle,draw=black,fill=gray,thin,inner sep=0pt,minimum size=1.2mm},scale=1.0]
\node[white] (w1)
at (-2,-2) {};
\node[white] (w2)
 at (0,-2) {};
\node[white] (w3)
 at (2,-2) {};
\node[white] (w4)
 at (-1,-1) {};
\node[white] (w5)
 at (1,-1) {};
\node[white] (w6)
 at (-2,0) {};
\node[white] (w7)
 at (0,0) {};
\node[white] (w8)
 at (2,0) {};
\node[white] (w9)
 at (-1,1) {};
\node[white] (w10)
 at (1,1) {};
\node[white] (w11)
 at (-2,2) {};
\node[white] (w12)
 at (0,2) {};
\node[white] (w13)
 at (2,2) {};

\node[black] (b1)
 at (-1,-2) {};
\node[black] (b2)
 at (1,-2) {};
\node[black] (b3)
 at (-2,-1) {};
\node[black] (b4)
 at (0,-1) {};
\node[black] (b5)
 at (2,-1) {};
\node[black] (b6)
 at (-1,0) {};
\node[black] (b7)
 at (1,0) {};
\node[black] (b8)
 at (-2,1) {};
\node[black] (b9)
 at (0,1) {};
\node[black] (b10)
 at (2,1) {};
\node[black] (b11)
 at (-1,2) {};
\node[black] (b12)
 at (1,2) {};

\node[gray] (m1)
 at (0,-1.5) {};
\node[gray] (m2)
 at (0.5,-1) {};
\node[gray] (m3)
 at (1,-0.5) {};
\node[gray] (m4)
 at (1.5,0) {};
\node[gray] (m5)
 at (1,0.5) {};
\node[gray] (m6)
 at (0.5,1) {};
\node[gray] (m7)
 at (0,1.5) {};
\node[gray] (m8)
 at (-0.5,1) {};
\node[gray] (m9)
 at (-1,0.5) {};
\node[gray] (m10)
 at (-1.5,0) {};
\node[gray] (m11)
 at (-1,-0.5) {};
\node[gray] (m12)
 at (-0.5,-1) {};

\draw[dashed] (w1) -- (b1) -- (w2) -- (b2) -- (w3);
\draw[dashed] (b3) -- (w4) -- (b4) -- (w5) -- (b5);
\draw[dashed] (w6) -- (b6) -- (w7) -- (b7) -- (w8);
\draw[dashed] (b8) -- (w9) -- (b9) -- (w10) -- (b10);
\draw[dashed] (w11) -- (b11) -- (w12) -- (b12) -- (w13);

\draw[dashed] (w1) -- (b3) -- (w6) -- (b8) -- (w11);
\draw[dashed] (b1) -- (w4) -- (b6) -- (w9) -- (b11);
\draw[dashed] (w2) -- (b4) -- (w7) -- (b9) -- (w12);
\draw[dashed] (b2) -- (w5) -- (b7) -- (w10) -- (b12);
\draw[dashed] (w3) -- (b5) -- (w8) -- (b10) -- (w13);

\draw (w2) -- (w5) -- (w8) -- (w10) -- (w12) -- (w9) -- (w6) -- (w4) -- (w2);
\draw (b4) -- (b7) -- (b9) -- (b6) -- (b4);

\draw[color=gray] (m1) -- (m2) -- (m3) -- (m4) -- (m5) -- (m6) -- (m7) -- (m8) -- (m9) -- (m10) -- (m11) -- (m12) -- (m1);

\coordinate[label=center:$W(P)$] (z1)  at (-0.2,-0.3) {};
\coordinate[label=center:$P$] (z2)  at (0.5,-1.25) {};
\coordinate[label=center:$B(P)$] (z3)  at (-0.85,-1.7) {};
\end{tikzpicture}
\endpgfgraphicnamed
\caption{Cycles $P$ on $X$, $B(P)$ on $\Gamma$, and $W(P)$ on $\Gamma^*$}
\label{fig:contours2}
\end{center}
\end{figure}

\begin{definition}
Let $\omega$ be a closed discrete differential of type $\Diamond$. For $1 \leq k \leq g$, we define its $a_k$\textit{-periods} $A_k:=\oint_{\alpha'_k} \omega$ and $b_k$\textit{-periods} $B_k:=\oint_{\beta'_k} \omega$ and its
\begin{align*}
 \textit{black } a_k\textit{-periods } A^B_k&:=2\int_{B\alpha'_k} \omega \;  \textnormal{ and } \; \textit{black } b_k\textit{-periods } B^B_k:=2\int_{B\beta'_k} \omega\\
 \textnormal{and its }\textit{white } a_k\textit{-periods } A^W_k&:=2\int_{W\alpha'_k} \omega \textnormal{ and } \textit{white } b_k\textit{-periods } B^W_k:=2\int_{W\beta'_k} \omega.
\end{align*}
\end{definition}

\begin{remark}
The reason for the factor of two is that to compute the black or white periods, we actually integrate $\omega$ on $\Gamma$ or $\Gamma^*$ and not on the medial graph $X$. Clearly, $2A_k=A_k^B+A_k^W$ and $2B_k=B_k^B+B_k^W$.
\end{remark}

\begin{lemma}
The periods of the closed discrete differential $\omega$ of type $\Diamond$ depend only on the homology classes $a_k$ and $b_k$, i.e., if $\alpha''_k, \beta''_k$, $1 \leq k \leq g$, are loops on $X$ that are in the homology classes $a_k$ and $b_k$, respectively, then
\begin{align*}
 \int_{a_k} \omega &:= A_k=\oint_{\alpha''_k} \omega, \int_{Ba_k} \omega :=A_k^B=2\int_{B\alpha''_k} \omega, \int_{Wa_k} \omega :=A_k^W=2\int_{W\alpha''_k} \omega;\\
  \int_{b_k} \omega &:= B_k=\oint_{\beta''_k} \omega, \int_{Bb_k} \omega :=B_k^B=2\int_{B\beta''_k} \omega, \int_{Wb_k} \omega :=B_k^W=2\int_{W\beta''_k} \omega.
\end{align*}
\end{lemma}

\begin{proof}
That $a$- and $b$-periods of a closed discrete one-form depend on the homology class only follows from discrete Stokes' Theorem~\ref{th:stokes}. For the other four cases, we use that $\omega$ induces discrete differentials on $\Gamma$ and $\Gamma^*$ in the obvious way since it is of type $\Diamond$. These differentials are closed in the sense that the integral along the black (or white) cycle around any white (or black) vertex of $\Lambda$ vanishes. Since the paths $B\alpha'_k$ and $B\alpha''_k$ on $\Gamma$ are both in the homology class $a_k$, $\int_{B\alpha'_k} \omega=\int_{B\alpha''_k} \omega$. The same reasoning applies for the other cases.
\end{proof}


\subsection{Discrete Riemann bilinear identity}\label{sec:RBI}

Again, let $\alpha_1, \ldots, \alpha_g, \beta_1, \ldots, \beta_g$ be smooth loops on the compact surface $\Sigma$ with base point $v_0 \in V(\Lambda)$ such that these loops cut out a fundamental $4g$-gon $F_g$. For the following two definitions, we follow \cite{BoSk12},but give a different proof for the discrete Riemann bilinear identity than \cite{BoSk12}.

\begin{definition}
For a loop $\alpha$ on $\Sigma$, let $d_{\alpha}$ denote the induced deck transformations on $\tilde{\Sigma}$, $\tilde{\Lambda}$, and $\tilde{X}$.
\end{definition}

\begin{definition}
$f:V(\tilde{\Lambda})\to \mC$ is \textit{multi-valued} with black periods $A_1^B, A_2^B, \ldots, A_g^B$, $B_1^B, B_2^B, \ldots, B_g^B \in \mC$ and white periods $A_1^W, A_2^W, \ldots, A_g^W$, $B_1^W, B_2^W, \ldots, B_g^W \in \mC$ if
\begin{align*}
f(d_{\alpha_k}b)=f(b)+A_k^B, \quad f(d_{\alpha_k}w)=f(w)+A_k^W,\quad f(d_{\beta_k}b)=f(b)+B_k^B,\quad f(d_{\beta_k}w)=f(w)+B_k^W
\end{align*}
for any $1\leq k \leq g$, each black vertex $b\in V(\tilde{\Gamma})$, and each white vertex $w \in V(\tilde{\Gamma}^*)$. 
\end{definition}

\begin{lemma}\label{lem:multivalued}
Let $f:V(\tilde{\Lambda})\to \mC$ be multi-valued. Then, $df$ defines a closed discrete one-form of type $\Diamond$ on the oriented edges of $X$ and $df$ has the same black and white periods as $f$. Conversely, if $\omega$ is a closed discrete differential of type $\Diamond$, then there is a multi-valued function $f:V(\tilde{\Lambda})\to \mC$ such that $df$ projects to $\omega$. If $\omega$ is discrete holomorphic, then $f$ is as well.
\end{lemma}
\begin{proof}
Let $\tilde{e}$ be an oriented edge of $\tilde{X}$. Discrete Stokes' Theorem~\ref{th:stokes} implies that $df(d_\alpha(\tilde{e}))=df(\tilde{e})$ for any loop $\alpha$ on $\Sigma$. In particular, $df$ is well-defined on the oriented edges of $X$. Closeness follows from $ddf=0$ by Proposition~\ref{prop:dd0}. Clearly, black and white periods of $f$ and $df$ are the same by definition of these periods.

Let $\tilde{\omega}$ be the lift of $\omega$ to $\tilde{X}$. Since the universal cover is simply-connected, it follows from Proposition~\ref{prop:primitive2} that there exists a discrete primitive $f:=\int \tilde{\omega} :V(\tilde{\Lambda})\to\mC$ that is discrete holomorphic if $f$ is.
\end{proof}
\begin{remark}
As a consequence, white and black periods of a closed discrete one-form of type $\Diamond$ are not determined by its periods.
\end{remark}

We are now ready to prove the following \textit{discrete Riemann bilinear identity}. 

\begin{theorem}\label{th:RBI}
Let $\omega$ and $\omega'$ be closed discrete differentials of type $\Diamond$. Let their black and white periods be given by $A_k^B, B_k^B, A_k^W, B_k^W$ and ${A'_k}^B, {B'_k}^B, {A'_k}^W, {B'_k}^W$, respectively, for $k=1,\ldots,g.$ Then, \[\iint_{F(X)} \omega \wedge \omega'=\frac{1}{2}\sum_{k=1}^g \left(A_k^B {B'_k}^W-B_k^B {A'_k}^W\right)+\frac{1}{2}\sum_{k=1}^g \left(A_k^W {B'_k}^B-B_k^W {A'_k}^B\right).\]
\end{theorem}

\begin{proof}
By Lemma~\ref{lem:multivalued}, there is a multi-valued function $f:V(\tilde{\Lambda})\to \mC$ such that $df=\omega$ with the same black and white periods as $\omega$ has. Let $F_v$ and $F_Q$ be faces of $X$ corresponding to $v\in V(\Lambda)$ and $Q\in V(\Diamond)\cong F(\Lambda)$. Consider any lifts of the star of $v$ and of $Q$ to $\tilde{\Lambda}$, and denote by $\tilde{F}_v$ and $\tilde{F}_Q$ the corresponding lifts of $F_v$ and $F_Q$ to $F(\tilde{X})$. By Theorem~\ref{th:derivation}, $\omega \wedge \omega'= d(f\omega')$, lifting $\omega,\omega'$ to $\tilde{X}$ and using that $\omega'$ is closed. So by discrete Stokes' Theorem~\ref{th:stokes}, \[\iint_{F} \omega \wedge \omega'=\oint_{\partial \tilde{F}} f\omega',\] where $F$ is either $F_v$ or $F_Q$. Note that the right hand side is independent of the chosen lift $\tilde{F}$ because $\omega'$ is closed. It follows that the statement above remains true when we integrate over $F(X)$ and the counterclockwise oriented boundary of any collection $\tilde{F}(X)$ of lifts of each a face of $X$ to $\tilde{X}$.

It remains to compute $\int_{\partial \tilde{F}(X)} f\omega'$. If $g=0$, then $\tilde{\Sigma}=\Sigma$ and $f$ is a complex function on $V(\Lambda)$. Furthermore, the boundary of $\tilde{F}(X)=F(X)$ is empty, so $\iint_{F(X)} \omega \wedge \omega'=0$ as claimed. In what follows, let $g>0$.

By definition, if $e=[\tilde{Q},\tilde{v}]$ is an edge of $\tilde{X}$ ($F(\tilde{\Lambda})\ni\tilde{Q}\sim \tilde{v} \in V(\tilde{\Lambda})$), then $\int_e f\omega'=f(\tilde{v})\int_e \omega'$. So we may consider $f$ as a function on $E(\tilde{X})$ defined by $f([\tilde{Q},\tilde{v}]):=f(\tilde{v})$. Then, $f:E(\tilde{X})\to\mC$ fulfills for any $k$:
\begin{align*}
f(d_{\alpha_k}[\tilde{Q},\tilde{v}])=f([\tilde{Q},\tilde{v}])+A_k^B &\textnormal{ and } f(d_{\beta_k}[\tilde{Q},\tilde{v}])=f([\tilde{Q},\tilde{v}])+B_k^B \textnormal{ if } \tilde{v}\in V(\tilde{\Gamma}),\\
f(d_{\alpha_k}[\tilde{Q},\tilde{v}])=f([\tilde{Q},\tilde{v}])+A_k^W &\textnormal{ and } f(d_{\beta_k}[\tilde{Q},\tilde{v}])=f([\tilde{Q},\tilde{v}])+B_k^W \textnormal{ if } \tilde{v}\in V(\tilde{\Gamma}^*).
\end{align*}
In this sense, $f$ is multi-valued on $E(\tilde{X})$ with black (white) periods defined on white (black) edges.

Since $f$ and $\omega'$ are now determined by topological data, we may forget the discrete complex structure of $\tilde{\Sigma}$ and can consider $\omega'$ and $f\omega'$ as functions on the oriented edges. Their evaluation on an edge $e$ will still be denoted by $\int_e$. Let $\tilde{\Sigma}'$ be the polyhedral surface that is given by $\tilde{X}$ requiring that all faces are regular polygons of side length one. Similarly, $\Sigma'$ is constructed. Now, $p$ induces a covering $p:\tilde{\Sigma}'\to \Sigma'$ in a natural way requiring that $p$ on each face is an isometry.

The homeomorphic images of the paths $\alpha_k,\beta_k$ are loops on $\Sigma'$ with the base point being somewhere inside the face $F_{v_0}$. Let us choose piecewise smooth paths on $\Sigma'$ with base point being the center of $F_{v_0}$ homotopic to the previous loops such that the new paths (that will be denoted the same) still cut out a fundamental $4g$-gon.

For $v \in V(\Lambda)$, consider the same subdivision of all the lifts of the regular polygon corresponding to $F_v$ into smaller polygonal cells induced by straight lines. All new edges get the same color as the original edges of $F_v$ had, i.e., the opposite color to the one of $v$. We extend $f$ on the new edges by $f(v)$. Obviously, the new function is still multi-valued with the same periods. We define the one-form $\omega'$ on the new edges consecutively by inserting straight lines. Each time an existing oriented edge $e$ is subdivided into two equally oriented parts $e'$ and $e''$, we define $\int_{e'}\omega'=\int_{e''}\omega':=\int_{e}\omega'/2$. On segments of the inserted line, we define $\omega'$ by the condition that it should remain closed. Defining a black (or white) $c$-period of $\omega'$ on the subdivided cellular decomposition as twice the discrete integral over all black (or white) edges of a closed path with homology $c$, we see that the black and white $a$- and $b$-periods of $\omega'$ are the same as before.

Now, let $F_Q$ be the square corresponding to $Q \in F(\Lambda)$. We consider a subdivision of $F_Q$ (and all its lifts) into smaller polygonal cells induced by straight lines parallel to the edges of the square, requiring in addition that all subdivision points on the edges of $F_Q$ coming from the previous subdivisions of $F_v$, $v\sim Q$, are part of it. A new edge is black (or white) if it is parallel to an original black (or white) edge of $X$. Any new edge $e'$ is of length $0<l\leq 1$ and $e'$ is parallel to an edge $e$ of $F_Q$. Since $\omega'$ is of type $\Diamond$, it coincides on parallel edges, so we can define $\int_{e'}\omega':=l\int_{e}\omega'$. By construction, the new discrete one-form $\omega'$ is closed, and its black and white periods do not change. $f$ is extended in such a way that if the new edge $e'$ is parallel to the edges $[Q,v]$ and $[Q,v']$, having distance $0\leq l\leq 1$ to $[Q,v]$ and distance $1-l$ to $[Q,v']$, then $f(e'):=(1-l)f([Q,v])+lf([Q,v'])$. $f$ is still multi-valued with the same periods.

If the subdivisions of faces $F_v$ and $F_Q$ are fine enough, then we find cycles homotopic to $\alpha_k,\beta_k$ on the edges of the resulting cellular decomposition $X'$ on $\Sigma$ in such a way that they still cut out a fundamental polygon with $4g$ vertices. Let us denote these loops by $\alpha_k,\beta_k$ as well.

By construction, $\oint_{\partial F} f\omega'$ equals the sum of all discrete contour integrals of $f\omega'$ around faces of the subdivision of the face $F$ of $X$. It follows that $\int_{\partial \tilde{F}(X)} f\omega'=\int_{\partial \tilde{F}(X')} f\omega'$ for any collection $\tilde{F}(X')$ of lifts of faces of $X'$, using that $\omega'$ is closed. Let us choose $\tilde{F}(X')$ in such a way that it builds a fundamental $4g$-gon whose boundary consists of lifts $\tilde{\alpha}_k,\tilde{\beta}_k$ of $\alpha_k,\beta_k$ and lifts $\tilde{\alpha}_k^{-1},\tilde{\beta}_k^{-1}$ of its reverses. Since interior edges of the polygon are traversed twice in both directions, they do not contribute to the discrete integral and we get \[\iint_{F(X)} \omega \wedge \omega'=\int_{\partial \tilde{F}(X')} f\omega'=\sum_{k=1}^g \left( \int_{\tilde{\alpha}_k} f\omega' +\int_{\tilde{\alpha}_k^{-1}} f\omega'\right)+\sum_{k=1}^g \left( \int_{\tilde{\beta}_k} f\omega' +\int_{\tilde{\beta}_k^{-1}} f\omega'\right).\]

Let $e$ be an edge of $\tilde{\alpha}_k$ and $e'$ the corresponding edge of $\tilde{\alpha}_k^{-1}$. Then, $d_{\beta_k}e=-e'$. Hence, $\omega'$ has opposite signs on $e$ and $e'$, and $f$ differs by $B_k^W$ on black edges and by $B_k^B$ on white edges. Therefore, \begin{align*} \int_{\tilde{\alpha}_k} f\omega' +\int_{\tilde{\alpha}_k^{-1}} f\omega'&=\int_{B\tilde{\alpha}_k} \left(f\omega' - (f+B_k^W)\omega'\right)+\int_{W\tilde{\alpha}_k} \left(f\omega' - (f+B_k^B)\omega'\right)\\&=-\frac{1}{2}B_k^W {A'_k}^B-\frac{1}{2}B_k^B {A'_k}^W. \end{align*}

If $e$ is an edge of $\tilde{\beta}_k$ and $e'$ the corresponding edge of $\tilde{\beta}_k^{-1}$, then $d_{\alpha_k^{-1}}e=-e'$. Thus, \[\int_{\tilde{\beta}_k} f\omega' +\int_{\tilde{\beta}_k^{-1}} f\omega'=\frac{1}{2}A_k^W {B'_k}^B+\frac{1}{2}A_k^B {B'_k}^W.\] Inserting the last two equations into the previous one gives the desired result.
\end{proof}

\begin{remark}
Note that as in the classical case, the formula is true for any canonical homology basis $\{a_1,\ldots,a_g,b_1,\ldots,b_g\}$, not necessarily the one we started with. The proof is essentially the same as in the smooth theory, see \cite{Gue14}.
\end{remark}

\begin{corollary}\label{cor:RBI2}
Let $\omega$ and $\omega'$ be closed discrete differentials of type $\Diamond$. Let their periods are given by $A_k, B_k$ and ${A'_k}, {B'_k}$, respectively, and assume that the black $a$-periods of $\omega,\omega'$ coincide with corresponding white $a$-periods. Then, \[\iint_{F(X)} \omega \wedge \omega'=\sum_{k=1}^g \left(A_k {B'_k}-B_k {A'_k}\right).\]
\end{corollary}


\section{Discrete harmonic and discrete holomorphic differentials}\label{sec:harmonic_holomorphic}

Throughout this section that aims in investigating discrete harmonic and discrete holomorphic differentials, let $(\Sigma,\Lambda,z)$ be a discrete Riemann surface. In Section~\ref{sec:Hodge_decomposition}, we state the discrete Hodge decomposition. Afterwards, we restrict to compact $\Sigma$ and compute the dimension of the space of discrete holomorphic differentials in Section~\ref{sec:harm_holo}. Discrete period matrices are introduced in Section~\ref{sec:period_matrices}. For Sections~\ref{sec:harm_holo} and~\ref{sec:period_matrices}, we therefore assume that $\Sigma$ is compact and of genus $g$. Let $\{a_1,\ldots,a_g,b_1,\ldots,b_g\}$ be a canonical basis of $H_1(\Sigma,\mZ)$ in this case.


\subsection{Discrete Hodge decomposition}\label{sec:Hodge_decomposition}

\begin{definition}
A discrete differential $\omega$ of type $\Diamond$ is \textit{discrete harmonic} if it is closed and \textit{co-closed}, i.e., $d\omega=0$ and $d\star \omega=0$ (or, equivalently, $\delta \omega=0$).
\end{definition}

\begin{lemma}\label{lem:harmonic_forms}
Let $\omega$ be a discrete differential of type $\Diamond$.
\begin{enumerate}
 \item $\omega$ is discrete harmonic if and only if for any $\Diamond_0\subseteq\Diamond$ forming a simply-connected closed region, there exists a discrete harmonic function $f:V(\Lambda_0)\to\mC$ such that $\omega=df$.
 \item Let $\Sigma$ be compact. Then, $\omega$ is discrete harmonic if and only if $\triangle \omega=0$.
\end{enumerate}
\end{lemma}
\begin{proof}
(i) Suppose that $\omega$ is discrete harmonic. Then, it is closed, so since $\Diamond_0$ forms a simply-connected closed region, Proposition~\ref{prop:primitive2} gives the existence of $f:V(\Lambda_0)\to\mC$ such that $\omega=df$ on oriented edges of $X_0$. Now, $\triangle f=\delta d f=\delta \omega =0$, so $f$ is discrete harmonic. Conversely, if $\omega=df$ locally, then $d\omega=ddf=0$ by Proposition~\ref{prop:dd0} (that is also locally true, see \cite{BoG15}) and $\delta \omega = \delta df=\triangle f=0$ by definition.

(ii) If $\omega$ is discrete harmonic, then $d\omega=\delta\omega=0$ implies $\triangle \omega=0$. Conversely, let $\triangle \omega=0$. Using that $\delta$ is the formal adjoint of $d$ on compact discrete Riemann surfaces by Proposition~\ref{prop:adjoint2}, \[0=\langle \triangle \omega, \omega \rangle= \langle d\omega, d\omega \rangle + \langle \delta\omega, \delta \omega \rangle.\] The right hand side vanishes only for $d\omega=\delta \omega=0$, so $\omega$ is closed and co-closed.
\end{proof}

The proof of the following \textit{discrete Hodge decomposition} follows the lines of the proof in the smooth theory given in the book \cite{FaKra80} of Farkas and Kra.

\begin{theorem}
Let $E,E^*$ denote the sets of \textit{exact} and \textit{co-exact} square integrable discrete differentials of type $\Diamond$, i.e., $E$ and $E^*$ consist of all $\omega=df$ and $\omega=\star df$, respectively, where $f:V(\Lambda) \to \mC$ and $\langle \omega,\omega\rangle <\infty$. Let $H$ be the set of square integrable discrete harmonic differentials. Then, we have an orthogonal decomposition $L_2(\Sigma,\Lambda,z)=E\oplus E^*\oplus H$.
\end{theorem}
\begin{proof}
Clearly, $E$ and $E^*$ are the closures of all exact and co-exact square integrable discrete differentials of type $\Diamond$ of compact support. Let $E^\perp$ and ${E^*}^\perp$ denote the orthogonal complements of $E$ and $E^*$ in $L_2(\Sigma,\Lambda,z)$. Then, $\omega \in E^\perp$ if and only if $\langle \omega, df\rangle=0$ for all $f:V(\Lambda)\to\mC$ of compact support. To compute the scalar product, we may restrict $\omega$ to a finite neighborhood of the support of $f$, so Proposition~\ref{prop:adjoint2} implies $0=\langle \omega,df\rangle=\langle \delta \omega,f\rangle$. It follows that $\delta \omega=0$. Thus, $E^\perp$ consists of all co-closed discrete differentials of type $\Diamond$. Similarly, ${E^*}^\perp$ is the space of all closed discrete differentials of type $\Diamond$. By Proposition~\ref{prop:dd0}, any (co-)exact discrete differential of type $\Diamond$ is (co-)closed, so we get an orthogonal decomposition $L_2(\Sigma,\Lambda,z)=E\oplus E^*\oplus H$, $H=E^\perp \cap {E^*}^\perp$ being the set of all discrete harmonic differentials.
\end{proof}


\subsection{Existence of certain discrete differentials}\label{sec:harm_holo}

First, we want to show that for any set of black and white periods there is a discrete harmonic differential with these periods. In \cite{Me07}, Mercat proved this statement by referring to a (discrete) Neumann problem. The proof given in \cite{BoSk12} used the finite-dimensional Fredholm alternative. Here, we give a proof based on the (discrete) Dirichlet energy.

\begin{theorem}\label{th:harmonic_existence}
Let $A_k^B,B_k^B,A_k^W, B_k^W$, $1\leq k \leq g$, be $4g$ given complex numbers. Then, there exists a unique discrete harmonic differential $\omega$ with these black and white periods.
\end{theorem}
\begin{proof}
Since periods are linear in the discrete differentials, it suffices to prove the statement for real periods. Let us consider the vector space of all multi-valued functions $f:V(\tilde{\Lambda}) \to \mR$ having the given black and white periods. For such a function $f$, $df$ is well-defined on $X$, as is the discrete Dirichlet energy $E_\Diamond(f)=\langle df, df\rangle$. By Lemma~\ref{lem:Dirichlet_boundary2}, the critical points of this functional are discrete harmonic functions, noting that $\triangle f$ is a function on $V(\Lambda)$. Since the discrete Dirichlet energy is convex, quadratic, and nonnegative, a minimum $f:V(\tilde{\Lambda}) \to \mR$ has to exist. By Lemma~\ref{lem:harmonic_forms}~(i), $\omega:=df$ is discrete harmonic and has the required periods by Lemma~\ref{lem:multivalued}.

Suppose that $\omega$ and $\omega'$ are two discrete harmonic differentials with the same black and white periods. Since $\omega-\omega'$ is closed, there is a multi-valued function $f:V(\tilde{\Lambda}) \to \mC$ such that $\omega-\omega'=df$ by Lemma~\ref{lem:multivalued}. But black and white periods of $f$ vanish, so $f$ is well-defined on $V(\Lambda)$ and discrete harmonic by Proposition~\ref{lem:harmonic_forms}~(i). By discrete Liouville's Theorem~\ref{th:Liouville}, $\omega-\omega'=df=0$.
\end{proof}

\begin{lemma}\label{lem:holo_harm}
Let $\omega$ be a discrete differential of type $\Diamond$.
\begin{enumerate}
 \item $\omega$ is discrete harmonic if and only if it can be decomposed as $\omega=\omega_1+\bar{\omega}_2$, where $\omega_1,\omega_2$ are discrete holomorphic differentials.
 \item $\omega$ is discrete holomorphic if and only if it can be decomposed as $\omega=\alpha+i\star\alpha$, where $\alpha$ is a discrete harmonic differential.
\end{enumerate}
\end{lemma}
\begin{proof}
 (i) Suppose that $\omega=\omega_1+\bar{\omega}_2$, where $\omega_1,\omega_2$ are discrete holomorphic. Then, $\omega$ is closed since $\omega_1,\omega_2$ are, and it is co-closed since $d\star\omega_k=-id\omega_k=0$ by Lemma~\ref{lem:Hodge_projection}. Thus, $\omega$ is discrete harmonic.
 
 Conversely, let $\omega$ be discrete harmonic. Then, we can write $\omega=p dz_v + q d\bar{z}_v$ in a chart $z_v$ around $v\in V(\Lambda)$, where $p,q$ are complex functions on the faces incident to $v$. Define $\omega_1:=p dz_v$ and $\omega_2:=\bar{q} dz_v$ in the chart $z_v$. By Lemma~\ref{lem:Hodge_projection}, $\omega_1,\omega_2$ are well defined on the whole discrete Riemann surface as the projections of $\omega$ onto the $\pm i$-eigenspaces of $\star$.

Since $\omega$ is closed, $0=d\omega|_{F_v}=\left(\partial_{\Diamond} q (v) - \bar{\partial}_{\Diamond} p(v)\right)\Omega^{z_v}_\Lambda$, so $\partial_{\Diamond} q (v) = \bar{\partial}_{\Diamond} p (v)$. Similarly, $d\star \omega|_{F_v}=0$ implies $\partial_{\Diamond} q (v) = -\bar{\partial}_{\Diamond} p (v)$. Thus, $\bar{\partial}_{\Diamond} p(v)=0=\partial_{\Diamond} q(v)$, i.e., $p,\bar{q}$ are discrete holomorphic in $v$. It follows that $\omega_1,\omega_2$ are discrete holomorphic.
 
 (ii) Suppose that $\omega=\alpha+i\star\alpha$. Then, $d\omega=0$ because $\alpha$ is closed and co-closed. In addition, we have $\star \omega=\star\alpha-i\alpha=-i\omega$. By Lemma~\ref{lem:Hodge_projection}, $\omega$ is discrete holomorphic. Conversely, for discrete harmonic $\omega$ we define $\alpha:=(\omega+\bar{\omega})/2$ that is discrete harmonic by (i) and that satisfies $\omega=\alpha+i\star\alpha$ by construction.
\end{proof}

\begin{corollary}\label{cor:dimension}
 The complex vector space $\mathcal{H}$ of discrete holomorphic differentials has dimension $2g$.
\end{corollary}

\begin{proof}
 Using that $\langle\omega_1,\bar{\omega_2}\rangle=\omega_1\wedge \star\omega_2=0$ for discrete holomorphic differentials $\omega_1,\omega_2$, Lemma~\ref{lem:holo_harm} implies that the space of discrete harmonic differentials $H$ is a direct orthogonal sum of the spaces of discrete holomorphic and discrete antiholomorphic one-forms, $\mathcal{H}$ and $\bar{\mathcal{H}}$. Due to Theorem~\ref{th:harmonic_existence}, $\textnormal{dim } H=4g$. Since $\mathcal{H}$ and $\bar{\mathcal{H}}$ are isomorphic, $\textnormal{dim } \mathcal{H}=2g$.
\end{proof}

\begin{remark}
As for the space of discrete harmonic differentials, the dimension of $\mathcal{H}$ is twice as high as the one of its classical counterpart due to the splitting of periods into black and white periods.
\end{remark}

\begin{lemma}\label{lem:holomorphic_periods}
Let $\omega\neq 0$ be a discrete holomorphic differential whose black and white periods are given by $A_k^B,B_k^B$ and $A_k^W,B_k^W$, $1\leq k \leq g$. Then, \[\im\left(\sum_{k=1}^g  \left(A_k^B \bar{B}_k^W+A_k^W \bar{B}_k^B \right)\right) <0.\]
\end{lemma}
\begin{proof}
Since $\omega$ is discrete holomorphic, $\omega$ and $\bar{\omega}$ are closed. Thus, we can apply the discrete Riemann Bilinear Identity~\ref{th:RBI} to them: \begin{align*}\iint_{F(X)} \omega \wedge \bar{\omega}= \frac{1}{2}\sum_{k=1}^g \left(A_k^B \bar{B}_k^W-B_k^B {\bar{A}}_k^W\right)+\frac{1}{2}\sum_{k=1}^g \left(A_k^W {\bar{B}}_k^B-B_k^W {\bar{A}}_k^B\right)=\sum_{k=1}^g i\im \left(A_k^B \bar{B}_k^W+A_k^W \bar{B}_k^B \right).\end{align*}
On the other hand, $\omega \wedge \bar{\omega}$ vanishes on faces $F_v$ of $X$ corresponding to vertices $v\in V(\Lambda)$ and in a chart $z_Q$ of $Q\in F(\Lambda)$, $\omega \wedge \bar{\omega}=|p|^2 \Omega^{z_Q}_\Diamond$ if $\omega|_{\partial F_Q}=p dz_Q$. Since $\omega \neq 0$, $p \neq 0$ for at least one $Q$ and \[\im\left(\sum_{k=1}^g  \left(A_k^B \bar{B}_k^W+A_k^W \bar{B}_k^B \right)\right)=\im\left(\iint_{F(X)} \omega \wedge \bar{\omega}\right)< 0. \qedhere\]
\end{proof}

\begin{corollary}\label{cor:periods_vanish}
Let $\omega$ be a discrete holomorphic differential.
\begin{enumerate}
\item If all black and white $a$-periods of $\omega$ vanish, then $\omega=0$. 
\item If all black and white periods of $\omega$ are real, then $\omega=0$. 
\end{enumerate}
\end{corollary}
\begin{proof}
If all black and white $a$-periods vanish or all black and white periods of $\omega$ are real, then \[\im\left(\sum_{k=1}^g  \left(A_k^B \bar{B}_k^W+A_k^W \bar{B}_k^B \right)\right) =0.\] In particular, $\omega=0$ by Lemma~\ref{lem:holomorphic_periods}.
\end{proof}

\begin{theorem}\label{th:holomorphic_existence}
Let $(\Sigma,\Lambda,z)$ be a compact discrete Riemann surface of genus $g$. 
\begin{enumerate}
\item For any $2g$ complex numbers $A_k^B,A_k^W$, $1\leq k\leq g$, there exists exactly one discrete holomorphic differential $\omega$ with these black and white $a$-periods.
\item For any $4g$ real numbers $\re\left(A_k^B\right),\re\left(B_k^B\right),\re\left(A_k^W\right),\re\left(B_k^W\right)$, there exists exactly one discrete holomorphic differential $\omega$ such that its black and white periods have these real parts.
\end{enumerate}
\end{theorem}
\begin{proof}
Let us consider the complex-linear map $P_1:\mathcal{H}\to\mC^{2g}$ that assigns to each discrete holomorphic differential its black and white $a$-periods and the real-linear map $P_2:\mathcal{H}\to\mR^{4g}$ that assigns to each discrete holomorphic differential the real parts of its black and white periods. By Corollary~\ref{cor:periods_vanish}, $P_1$ and $P_2$ are injective. By Corollary~\ref{cor:dimension}, $\mathcal{H}$ has complex dimension $2g$, so $P_1$ and $P_2$ have to be surjective.
\end{proof}


\subsection{Discrete period matrices}\label{sec:period_matrices}

Discrete period matrices in the special case of real weights $\rho_Q$ were already studied by Mercat in \cite{Me01b, Me07}. In \cite{BoSk12}, a proof of convergence of discrete period matrices to their continuous counterparts was given and the case of complex weights was sketched.

By Theorem~\ref{th:holomorphic_existence}, there exists exactly one discrete holomorphic differential with prescribed black and white $a$-periods. Having a limit of finer and finer quadrangulations of a Riemann surface in mind, it is natural to demand that black and white $a$-periods coincide.

\begin{definition}
The unique set of $g$ discrete holomorphic differentials $\omega_k$ that satisfies for all $1\leq j,k \leq g$ the equation $2\int_{Ba_j}\omega_k=2\int_{Wa_j}\omega_k=\delta_{jk}$ is called \textit{canonical}. The $(g\times g)$-pmatrix $\left(\Pi_{jk}\right)_{j,k=1}^g$ with entries $\Pi_{jk}:=\int_{b_j}\omega_k$ is the \textit{discrete period pmatrix} of the discrete Riemann surface $(\Sigma,\Lambda,z)$.
\end{definition}

The definition of the discrete period pmatrix as the arithmetic mean of black and white periods was already given in \cite{BoSk12}, adapting Mercat's definition in \cite{Me01b,Me07}. In our notation with discrete differentials defined on the medial graph it becomes clear why this is a natural choice. Still, it is reasonable to consider black and white periods separately to encode all possible information. We end up with the same matrices Mercat defined in \cite{Me01b,Me07}.

\begin{definition}
Let $\omega_k^B$, $1\leq k \leq g$, be the unique discrete holomorphic differential with black $a_j$-period $\delta_{jk}$ and vanishing white $a$-periods. Furthermore, let $\omega_k^W$, $1\leq k \leq g$, be the unique discrete holomorphic differential with white $a_j$-period $\delta_{jk}$ and vanishing black $a$-periods. The basis of these $2g$ discrete differentials is called the \textit{canonical basis (of discrete holomorphic differentials)}.

We define the $(g\times g)$-matrices $\Pi^{B,B},\Pi^{W,B},\Pi^{B,W},\Pi^{W,W}$ with entries \begin{align*}\Pi^{B,B}_{jk}:=2\int_{Bb_j}\omega^B_k, \quad \Pi^{W,B}_{jk}:=2\int_{Wb_j}\omega^B_k,\quad \Pi^{B,W}_{jk}:=2\int_{Bb_j}\omega^W_k, \quad \Pi^{W,W}_{jk}:=2\int_{Wb_j}\omega^W_k.\end{align*} 

The \textit{complete discrete period pmatrix} is the $(2g\times 2g)$-pmatrix defined by \[\tilde{\Pi}:=\left( \begin{matrix} \Pi^{B,W} & \Pi^{B,B}\\ \Pi^{W,W} & \Pi^{W,B}\end{matrix}\right).\]
\end{definition}
\begin{remark}
Note that $\omega_k=\omega_k^W+\omega_k^B$ implies that $\Pi=(\Pi^{B,W} + \Pi^{B,B}+ \Pi^{W,W} + \Pi^{W,B})/2$.
\end{remark}
\begin{example}
In the example of a bipartitely quadrangulated flat torus $\Sigma=\mC/(\mZ+\mZ\tau)$ of modulus $\tau \in \mC$ with $\im \tau >0$, the classical period of the Riemann surface $\Sigma$ is $\tau$. In the discrete setup, $dz$ is globally defined and discrete holomorphic. It follows that the discrete period $\Pi$ equals the $b$-period of $dz$ that is $\tau$. Thus, discrete and smooth period coincide in this case.
\end{example}

\begin{remark}
Although the black and white $a$-periods of the canonical set of discrete holomorphic differentials coincide by definition, the black and white $b$-periods must not in general. A counterexample was given in \cite{BoSk12}, namely the bipartite quad-decomposition of a torus induced by the triangulation given by identifying opposite sides of the base of the side surface of a regular square pyramid and its dual.
\end{remark}

\begin{theorem}\label{th:period_pmatrix}
Both the discrete period pmatrix $\Pi$ and the complete discrete period pmatrix $\tilde{\Pi}$ are symmetric and their imaginary parts are positive definite.
\end{theorem}
\begin{proof}
Let $\{\omega_1,\ldots,\omega_g\}$ be the canonical set of discrete holomorphic differentials used to compute $\Pi$. By looking at the coordinate representations, $\omega_j \wedge \omega_k=0$ for all $j,k$. Inserting this into the discrete Riemann Bilinear Identity~\ref{th:RBI}, the periods of $\omega:=\omega_j$ and $\omega':=\omega_k$ satisfy \begin{align*}0&=\sum_{l=1}^g \left(A_l^B {B'_l}^W-B_l^B {A'_l}^W\right)+\sum_{l=1}^g \left(A_l^W {B'_l}^B-B_l^W {A'_l}^B\right)={B'_j}^W-B_k^B+{B'_j}^B-B_k^W=2\Pi_{jk}-2\Pi_{kj}.\end{align*}

Applying the same arguments to discrete differentials of the canonical basis $\{\omega_1^W,\ldots,\omega_g^W,\omega_1^B,\ldots,\omega_g^B\}$, \[(\Pi^{B,W})^T=\Pi^{B,W} \textnormal{ and } (\Pi^{W,B})^T=\Pi^{W,B}\] if we apply the discrete Riemann Bilinear Identity~\ref{th:RBI} to all pairs $\omega_j^W,\omega_k^W$ and $\omega_j^B,\omega_k^B$, respectively. Considering pairs $\omega_j^W,\omega_k^B$ yields $(\Pi^{B,B})^T=\Pi^{W,W}.$ Thus, $\Pi$ and $\tilde{\Pi}$ are symmetric.

Let $\alpha=(\alpha_1,\ldots,\alpha_g)^T$ be a nonzero real column vector. Applying Lemma~\ref{lem:holomorphic_periods} to the discrete holomorphic differential $\omega:=\sum_{k=1}^g \alpha_k \omega_k$ with black and white $a_k$-period $\alpha_k$ yields \[0>\im\left(\sum_{k=1}^g \left(\alpha_k \sum_{j=1}^g\alpha_j 2\overline{\Pi}_{kj}\right)\right)=-2\im\left(\alpha^T \Pi \alpha\right).\] Hence, $\im (\Pi)$ is positive definite. Similarly, $\im (\tilde{\Pi})$ is positive definite.
\end{proof}

Since black and white $b$-periods of a discrete holomorphic differential do not have to coincide even if their black and white $a$-periods do, the discrete period matrices do not change similarly to the classical theory if another canonical homology basis is chosen, but the complete discrete period matrices do.

\begin{proposition}\label{prop:transformation}
The complete discrete period matrices $\tilde{\Pi}$ and $\tilde{\Pi}'$ corresponding to the canonical homology bases $\left\{a,b\right\}$ and $\left\{a',b'\right\}$, respectively, are related by \[\tilde{\Pi}'=\left(\tilde{C}+\tilde{D}\tilde{\Pi}\right)\left(\tilde{A}+\tilde{B}\tilde{\Pi}\right)^{-1}.\] Here, the two canonical bases are related by $\left( \begin{smallmatrix} a' \\ b'\end{smallmatrix}\right)=\left( \begin{smallmatrix} A & B\\ C & D\end{smallmatrix}\right) \left( \begin{smallmatrix} a \\ b\end{smallmatrix}\right)$ and the $(2g\times 2g)$-matrices $\tilde{A},\tilde{B},\tilde{C},\tilde{D}$ are given by $\tilde{A}:=\left( \begin{smallmatrix} A & 0\\ 0 & A\end{smallmatrix}\right),\tilde{B}:=\left( \begin{smallmatrix} B & 0\\ 0 & B\end{smallmatrix}\right),\tilde{C}:=\left( \begin{smallmatrix} C & 0\\ 0 & C\end{smallmatrix}\right),\tilde{D}:=\left( \begin{smallmatrix} D & 0\\ 0 & D\end{smallmatrix}\right).$
\end{proposition}
\begin{proof}
Let $\omega = (\omega_1^W,\ldots,\omega_g^W,\omega_1^B,\ldots,\omega_g^B)$ be the canonical basis of discrete holomorphic differentials corresponding to $(a,b)$. Labeling the columns of the matrices by discrete differentials and their rows by first all white and then all black cycles we get \[\int_{Wa',Ba'} \omega=\tilde{A}+\tilde{B}\tilde{\Pi}, \quad \int_{Wb',Bb'} \omega=\tilde{C}+\tilde{D}\tilde{\Pi}.\]

Thus, the canonical basis $\omega'$ corresponding to $(a',b')$ is given by $\omega'=\omega\left(\tilde{A}+\tilde{B}\tilde{\Pi}\right)^{-1}$ and \[\tilde{\Pi}'=\int_{Wb',Bb'} \omega'=\int_{Wb',Bb'} \omega \left(\tilde{A}+\tilde{B}\tilde{\Pi}\right)^{-1} =\left(\tilde{C}+\tilde{D}\tilde{\Pi}\right)\left(\tilde{A}+\tilde{B}\tilde{\Pi}\right)^{-1}.\qedhere\]
\end{proof}


\section{Discrete theory of Abelian differentials}\label{sec:Abelian_theory}

After introducing discrete Abelian differentials in Section~\ref{sec:Abelian_differentials} and discussing several properties of them, the aim of Section~\ref{sec:RR} is to state and prove the discrete Riemann-Roch Theorem~\ref{th:Riemann_Roch}. We conclude this chapter by discussing discrete Abel-Jacobi maps in Section~\ref{sec:Abel}. Throughout this section, we consider a compact discrete Riemann surface $(\Sigma,\Lambda,z)$ of genus $g$. Let $\{a_1,\ldots,a_g,b_1,\ldots,b_g\}$ be a canonical basis of its homology, $\{\omega_1, \ldots, \omega_g\}$ the canonical set and $\{\omega_1^B, \omega_1^W \ldots, \omega_g^B, \omega_g^W\}$ the canonical basis of discrete holomorphic differentials.


\subsection{Discrete Abelian differentials}\label{sec:Abelian_differentials}

\begin{definition}
A discrete differential $\omega$ of type $\Diamond$ is said to be a \textit{discrete Abelian differential}. For a vertex $v\in V(\Lambda)$ and its corresponding face $F_v \in F(X)$, the \textit{residue} of $\omega$ at $v$ is defined as \[\textnormal{res}_v (\omega) := \frac{1}{2\pi i} \oint_{\partial F_v} \omega.\]
\end{definition}
\begin{remark}
By definition, the discrete integral of a discrete differential of type $\Diamond$ around a face $F_Q$ corresponding to $Q \in V(\Diamond)$ is always zero. For this reason, a residue at faces $Q \in V(\Diamond)$ is not defined.
\end{remark}

\begin{proposition}\label{prop:residue}
Discrete residue theorem: Let $\omega$ be a discrete Abelian differential. Then, the sum of all residues of $\omega$ at black vertices vanishes as well as the sum of all residues of $\omega$ at white vertices: \[\sum_{b \in V(\Gamma)} \textnormal{res}_b (\omega)=0=\sum_{w \in V(\Gamma^*)} \textnormal{res}_w (\omega).\]
\end{proposition}
\begin{proof}
Since $\omega$ is of type $\Diamond$, $\int_{[Q,b_-]} \omega=-\int_{[Q,b_+]} \omega$ if $b_-,b_+$ are two black vertices incident to a quadrilateral $Q \in F(\Lambda)$ and $[Q,b_-]$ and $[Q,b_+]$ are oriented in such a way that they go clockwise around $F_Q$. Equivalently, they are oriented in such a way that they go counterclockwise around $F_{b_-}$ and $F_{b_+}$, respectively. It follows that the sum of all residues of $\omega$ at black vertices can be arranged in pairwise canceling contributions. Thus, the sum is zero. Similarly, $\sum_{w \in V(\Gamma^*)} \textnormal{res}_w (\omega)=0$.
\end{proof}

\begin{definition}
Let $\omega$ be a discrete Abelian differential, $v\in V(\Lambda)$, $Q \in V(\Diamond)$, and $F_Q$ the face of $X$ corresponding to $Q$. If $\omega$ has a nonzero residue at $v$, then $v$ is a \textit{simple pole} of $\omega$. If $z_Q$ is a chart of $Q$ and $\omega|_{\partial F_Q}$ is not of the form $p dz_Q$, $p\in \mC$, then $Q$ is a \textit{double pole} of $\omega$. If $\omega|_{\partial F_Q}=0$, then $Q$ is a \textit{zero} of $\omega$.
\end{definition}

\begin{remark}
To say that quadrilaterals $Q$ where $\omega \neq pdz_Q$ are double poles of $\omega$ is well motivated. In \cite{BoG15}, the existence of functions $K_Q$ on $V(\Lambda)$ that are discrete holomorphic at all but one fixed face $Q \in V(\Diamond)$ was shown. These functions appeared in the discrete Cauchy's integral formulae and model $z^{-1}$ besides its asymptotics. Similarly, $\partial_\Lambda K_Q$ models $-z^{-2}$. Now, $dK_Q$ should be like $-z^{-2} dz$, modeling a double pole at $Q$. By construction, $dK_Q$ is a discrete Abelian differential that is of the form $pdz_{Q'}$ in any chart $z_{Q'}$ around a face $Q' \neq Q$. But in a chart $z_Q$, $dK_Q=pdz_Q+qd\bar{z}_Q$ with $q\neq 0$.
\end{remark}

\begin{definition}
Let $\omega$ be a discrete Abelian differential. If $\omega$ is discrete holomorphic, then we say that $\omega$ is a \textit{discrete Abelian differential of the first kind}. If $\omega$ is not discrete holomorphic, but all its residues vanish, then it is a \textit{discrete Abelian differential of the second kind}. A discrete Abelian differential whose residues do not vanish identically is said to be a \textit{discrete Abelian differential of the third kind}.
\end{definition}

As in the classical setup, there exists a set of normalized discrete Abelian differentials with certain prescribed poles and residues that can be normalized such that their $a$-periods vanish. In the case of a Delaunay-Voronoi quadrangulation, the existence of corresponding normalized discrete Abelian integrals of the second kind and discrete Abelian differentials of the third kind was shown in \cite{BoSk12}. Our proofs will be similar, but in addition, we obtain the existence of certain discrete Abelian differentials of the second kind as a corollary. The computation of the $b$-periods of the normalized discrete Abelian differentials of the third kind is also new.

\begin{proposition}\label{prop:existence_third}
Let $v,v' \in V(\Gamma)$ or $v,v' \in V(\Gamma^*)$. Then, there exists a discrete Abelian differential of the third kind $\omega$ whose only poles are at $v$ and $v'$ and whose residues are $\textnormal{res}_\omega(v)=-\textnormal{res}_\omega(v')=1$. Any two such discrete differentials differ just by a discrete holomorphic differential.
\end{proposition}
\begin{proof}
Clearly, the difference of two discrete Abelian differentials of the third kind with equal residues and no double poles has no poles at all, so it is discrete holomorphic.

Let $V$ be the vector space of all discrete Abelian differentials that have no double poles. For any $Q \in V(\Diamond) \cong F(\Lambda)$, we choose one chart $z_Q$. By definition, each $\omega \in V$ is of the form $pdz_Q$ at $Q$. Conversely, any function $p:V(\Diamond)\to\mC$ defines by $pdz_Q$ a discrete Abelian differential that has no double poles. Thus, the complex dimension of $V$ equals $|F(\Lambda)|$.

Now, let $W$ be the image in $\mC^{|V(\Lambda)|}$ of the linear map $\textnormal{res}$ that assigns to each $\omega \in V$ all its residues at vertices of $\Lambda$. By Proposition~\ref{prop:residue}, the residues at all black points sum up to zero as well as all residues at white vertices. Thus, the complex dimension of $W$ is at most $|V(\Lambda)|-2$. Since $\Lambda$ is a quad-decomposition, $|V(\Lambda)|-2=|F(\Lambda)|-2g$. Therefore, the dimension of $W$ is at most $|F(\Lambda)|-2g$.

On the other hand, the dimension of $W$ equals $|F(\Lambda)|$ minus the dimension of the kernel of the map $\textnormal{res}$. But if $\omega \in V$ has vanishing residues, then it is discrete holomorphic. Due to Corollary~\ref{cor:dimension}, the space of discrete holomorphic differentials is $2g$-dimensional. For this reason, $\textnormal{dim }W=|F(\Lambda)|-2g=|V(\Lambda)|-2$. In particular, we can find a discrete Abelian differential without double poles for any prescribed residues that sum up to zero at all black and at all white vertices.
\end{proof}

\begin{corollary}\label{cor:existence_second}
Given a quadrilateral $Q \in F(\Lambda)$ and a chart $z_Q$, there exists a unique discrete Abelian differential of the second kind that is of the form \[pdz_Q-\frac{\pi}{2\textnormal{area}(z_Q(F_Q))} d\bar{z}_Q\] in the chart $z_Q$, that has no other poles, and whose black and white $a$-periods vanish. This discrete differential is denoted by $\omega_Q$. Here, $\textnormal{area}(z_Q(F_Q))$ denotes the Euclidean area of the parallelogram $z_Q(F_Q)$.
\end{corollary}
\begin{proof}
Consider the discrete Abelian differential of the third kind $\omega$ that is given by the local representation $-\pi/\left(2\textnormal{area}(z_Q(F_Q))\right)d\bar{z}_Q$ at the four edges of $F_Q$ and zero everywhere else. Its only poles other than $Q$ are at the four vertices incident to $Q$ and since $\omega$ is of type $\Diamond$, residues at opposite vertices are equal up to sign. Using Proposition~\ref{prop:existence_third} twice, we can find a discrete Abelian differential $\omega'$ that has no double poles and whose residues equal the residues of $\omega$. To get vanishing black and white $a$-periods, Theorem~\ref{th:holomorphic_existence} allows us to add a suitable discrete holomorphic differential $\omega''$ such that $\omega_Q:=\omega-\omega'+\omega''$ is what we are looking for. Since the difference of two such discrete differentials has vanishing black and white $a$-periods, uniqueness follows by Corollary~\ref{cor:periods_vanish}.
\end{proof}

\begin{remark}
As in the classical case, $\omega_Q$ depends on the choice of the chart $z_Q$. In our setting, the coefficient of $d\bar{z}_Q$ of $\omega_Q$ equals $-\bar{\partial}_\Lambda K_Q(Q) = -\pi /\left(2 \textnormal{area}(z_Q(F_Q))\right)$.
\end{remark}

\begin{lemma}\label{lem:property_second}
Let $Q\neq Q' \in F(\Lambda)$ and let $\omega_Q,\omega_{Q'}$ be the discrete Abelian differentials of the second kind corresponding to the charts $z_Q,z_{Q'}$. Define complex numbers $\alpha, \beta$ in such a way that $\omega_Q=\alpha dz_{Q'}$ on the four edges of $F_{Q'}$ and $\omega_{Q'}=\beta dz_{Q}$ on the four edges of $F_{Q}$. Then, $\alpha=\beta$.
\end{lemma}
\begin{proof}
By definition, $\omega_Q$ and $\omega_{Q'}$ are closed discrete differentials whose black and white $a$-periods vanish. So by the discrete Riemann Bilinear Identity~\ref{th:RBI}, $\iint_{F(X)} \omega_Q \wedge \omega_{Q'}=0$. Since $\omega_Q$ and $\omega_{Q'}$ have no pole at a face of $X$ corresponding to a quadrilateral $Q''\neq Q,Q'$, $\left(\omega_Q \wedge \omega_{Q'}\right)|_{F_{Q''}}=0$. Hence,
\begin{align*}
0&=\iint_{F(X)}\omega_Q \wedge \omega_{Q'}=-\iint_{F_Q}\omega_{Q'} \wedge \omega_Q+\iint_{F_{Q'}}\omega_Q \wedge \omega_{Q'}\\
&=\frac{\beta\pi}{2\textnormal{area}(z_Q(F_Q))}\iint_{F_Q}dz_Q\wedge d\bar{z}_Q-\frac{\alpha\pi}{2\textnormal{area}(z_{Q'}(F_{Q'}))}\iint_{F_{Q'}}dz_{Q'}\wedge d\bar{z}_{Q'}=-2\pi i (\beta-\alpha). \qedhere\end{align*}
\end{proof}

\begin{proposition}\label{prop:periods_second}
Let $Q \in F(\Lambda)$ and let $\omega_Q$ be the discrete Abelian differential of the second kind corresponding to the chart $z_Q$. Suppose that $\omega_k|_{\partial F_Q}=\alpha_k dz_Q$ for $k=1,\ldots,g$. Then, $\int_{b_k} \omega_Q=2\pi i \alpha_k$.
\end{proposition}
\begin{proof}
In a chart $z_{Q'}$ of a face $Q' \neq Q$, $\omega_k$ and $\omega_Q$ are both of the form $pdz_{Q'}$, so $\omega_k \wedge \omega_Q$ vanishes at $F_{Q'}$. It follows from the discrete Riemann Bilinear Identity~\ref{th:RBI} applied to $\omega_k$ and $\omega_Q$ that \[\int_{b_k} \omega_Q=\int_{W b_k} \omega_Q+\int_{B b_k} \omega_Q=\iint_{F(X)} \omega_k \wedge \omega_Q=\frac{-\alpha_k\pi}{2\textnormal{area}(z_Q(F_Q))}\iint_{F_Q}  dz_Q \wedge d\bar{z}_Q=2\pi i \alpha_k\] since black and white $a$-periods of $\omega_Q$ vanish.
\end{proof}

Since discrete Abelian differentials of the third kind have residues, periods are not well-defined. However, periods of the discrete Abelian differentials constructed in Proposition~\ref{prop:existence_third} are defined modulo $2\pi i$. To normalize them, we think of $a_k,b_k$ as given closed curves $\alpha'_k,\beta'_k$ on $X$.

\begin{definition}
Let $\alpha'_k,\beta'_k$, $1\leq k \leq g$, be cycles on $X$ in the homotopy classes $a_k,b_k$. Let $v,v' \in V(\Gamma)$ or $v,v' \in V(\Gamma^*)$. Then, $\omega_{vv'}$ denotes the unique discrete Abelian differential whose integrals along $\alpha'_k,\beta'_k$ are zero, whose nonzero residues are given by $\textnormal{res}_\omega(v)=-\textnormal{res}_\omega(v')=1$, and that has no further poles.
\end{definition}

\begin{definition}
Let $R$ be an oriented path on $\Gamma$ or $\Gamma^*$ and $\omega$ a discrete Abelian differential. To each oriented edge in $R$ we choose one of the corresponding parallel edges of $X$ and orient it the same. By $R_X$, we denote the resulting one-chain on $X$. Then, $\int_R \omega:=2\int_{R_X}\omega$.
\end{definition}

\begin{proposition}\label{prop:periods_third}
Let $v,v' \in V(\Gamma)$ or $v,v' \in V(\Gamma^*)$. Suppose that the cycles $\alpha'_k,\beta'_k$ on $X$ are homotopic to closed paths $\alpha_k,\beta_k$ on $\Sigma$ cutting out a fundamental polygon with $4g$ vertices on the surface $\Sigma \backslash \{v,v'\}$. In addition, let $R$ be an oriented path on $\Gamma$ or $\Gamma^*$ from $v'$ to $v$ that does not intersect any of the curves $\alpha_1,\ldots,\alpha_g,\beta_1,\ldots,\beta_g$. Then, \[\int_{b_k}\omega_{vv'}=2\pi i \int_R \omega_k.\]
\end{proposition}
\begin{proof}
On the one hand, $\omega_k \wedge \omega_{vv'}=0$ since both discrete Abelian differentials are of the form $pdz_Q$ in any chart $z_Q$. On the other hand, we can find a discrete holomorphic multi-valued function $f:V(\tilde{\Lambda})\to \mC$ such that $df=\omega_k$ by Lemma~\ref{lem:multivalued}. Since $d$ is a derivation by Theorem~\ref{th:derivation}, \[0=\omega_k \wedge \omega_{vv'}= d(f\omega_{vv'})-fd\omega_{vv'}\] is true if $\omega_k,\omega_{vv'}$ are lifted to $\tilde{X}$. Now, choose a collection $\tilde{F}(X)$ of lifts of all faces of $X$ to $\tilde{X}$ such that the corresponding lifts $\tilde{v}$ and $\tilde{v}'$ of $v$ and $v'$ are connected by a lift of $R$ in $\tilde{\Gamma}$ or $\tilde{\Gamma}^*$. It is not necessary that all faces of $\tilde{X}$ intersecting the lift of $R$ are contained in $\tilde{F}(X)$. Due to discrete Stokes' Theorem~\ref{th:stokes}, $d\omega_{vv'}=0$ on all lifts of faces $F_Q$ corresponding to $Q\in V(\Diamond)$ or faces $F_{v''}$ corresponding to a vertex $v''\neq v,v'$. Using $\textnormal{res}_\omega(v)=-\textnormal{res}_\omega(v')=1$, discrete Stokes' Theorem~\ref{th:stokes} gives \[\int_{\partial \tilde{F}(X)} f\omega_{vv'}=\iint_{\tilde{F}(X)} d(f\omega_{vv'})=\iint_{\tilde{F}(X)} fd\omega_{vv'}=2\pi i\left(f(\tilde{v})-f(\tilde{v}')\right)=2\pi i\int_{R} \omega_k.\]

The left hand side can be calculated in exactly the same way as in the proof of the discrete Riemann Bilinear Identity~\ref{th:RBI}. The only essential difference is that when we extend $\omega_{vv'}$ to a subdivision of lifts $\tilde{F}_{v}$ or $\tilde{F}_{v'}$, the extended one-form shall have zero residues at all new faces but one containing $v$ or $v'$, where it should remain $1$ or $-1$. As a result, we obtain $\int_{b_k}\omega_{vv'}$, observing that almost all black and white $a$-periods of $f$ and $\omega_{vv'}$ vanish.
\end{proof}

\begin{remark}
Results analogous to Propositions~\ref{prop:periods_second} and~\ref{prop:periods_third} are true for the black and white $b$-periods of $\omega_Q$ and $\omega_{vv'}$, replacing $\omega_k$ by $\omega_k^B$ or $\omega_k^W$.
\end{remark}

\begin{proposition}\label{prop:basis}
Let a chart $z_Q$ to each $Q \in F(\Lambda)$ be given. Fix $b \in V(\Gamma)$ and $w \in V(\Gamma^*)$. Then, the normalized discrete Abelian differentials of the first kind $\omega_k^B$ and $\omega_k^W$, $k=1,\ldots,g$, of the second kind $\omega_Q$, $Q \in F(\Lambda)$, and of the third kind $\omega_{bb'}$ and $\omega_{ww'}$, $b'\neq b$ being black and $w'\neq w$ being white vertices, form a basis of the space of discrete Abelian differentials.
\end{proposition}
\begin{proof}
Linear independence is clear. Given any discrete Abelian differential, we can first use the $\omega_Q$ to eliminate all double poles. For the resulting discrete Abelian differential we can find linear combinations of $\omega_{bb'}$ and  $\omega_{ww'}$ that have the same residues at black and white vertices, respectively. We end up with a discrete holomorphic differential that can be represented by a linear combination of the $2g$ discrete differentials $\omega_k^B$ and $\omega_k^W$.
\end{proof}


\subsection{Divisors and the discrete Riemann-Roch Theorem}\label{sec:RR}

We generalize the notion of divisors on a compact discrete Riemann surface $(\Sigma,\Lambda,z)$ of genus $g$ and the discrete Riemann-Roch theorem given in \cite{BoSk12} to general quad-decompositions. In addition, we define double poles of discrete Abelian differentials and double values of functions $f:V(\Lambda)\to\mC$.

\begin{definition}
A \textit{divisor} $D$ is a formal linear combination \[D=\sum_{j=1}^M m_j v_j + \sum_{k=1}^N n_k Q_k,\] where $m_j\in \left\{-1,0,1\right\}$, $v_j \in V(\Lambda)$, $n_k \in \left\{-2,-1,0,1,2\right\}$, and $Q_k \in V(\Diamond)$.

$D$ is \textit{admissible} if even $m_j\in \left\{-1,0\right\}$ and $n_k \in \left\{-2,0,1\right\}$. Its \textit{degree} is defined as \[\deg D:=\sum_{j=1}^M m_j + \sum_{k=1}^N \textnormal{sign}(n_k).\] $D\geq D'$ if the formal sum $D-D'$ is a divisor whose coefficients are all nonnegative.
\end{definition}
\begin{remark}
Note that double points just count once in the degree. The reason is that these points correspond to double values and not to double zeroes of a discrete meromorphic function. Concerning discrete Abelian differentials, a double pole does not include a simple pole and therefore counts once.
\end{remark}

As noted in \cite{BoSk12}, divisors on a discrete Riemann surface do not form an Abelian group. One of the reasons is that the pointwise product of discrete holomorphic functions does not need to be discrete holomorphic itself, another one is the asymmetry of point spaces. Whereas discrete meromorphic functions will be defined on $V(\Lambda)$, discrete Abelian differentials are essentially defined by complex functions on $V(\Diamond)$, supposed that a chart for each quadrilateral is fixed.

\begin{definition}
Let $f:V(\Lambda) \to \mC$, $v \in V(\Lambda)$, and $Q \in F(\Lambda)\cong V(\Diamond)$. $f$ is called \textit{discrete meromorphic}.
\begin{itemize}
 \item $f$ has a \textit{zero} at $v$ if $f(v)=0$.
 \item $f$ has a \textit{simple pole} at $Q$ if $df$ has a double pole at $Q$.
 \item $f$ has a \textit{double value} at $Q$ if $df|_{\partial F_Q}=0$.
\end{itemize}

If $f$ has zeroes $v_1,\ldots,v_M \in V(\Lambda)$, double values $Q_1,\ldots,Q_N \in V(\Diamond)$, and poles $Q'_1,\ldots,Q'_{N'} \in V(\Diamond)$, then its \textit{divisor} is defined as \[(f):=\sum_{j=1}^M v_j + \sum_{k=1}^N 2Q_k - \sum_{k'=1}^{N'} Q'_{k'}.\] 
\end{definition}

\begin{remark}
Note that in the smooth setting, a double value of a smooth function $f$ is a point where $f-c$ has a double zero for some constant $c$. In the discrete setup, a double value at a quadrilateral $Q$ implies that the values of the discrete function $f$ at both black vertices coincide as well as at the two white vertices of $Q$. In this sense, double values are separated from the points where the function is evaluated.
\end{remark}

\begin{definition}
Let $\omega$ be a discrete Abelian differential. If $\omega$ has zeroes $Q_1,\ldots,Q_N \in V(\Diamond)$, double poles at $Q'_1,\ldots,Q'_{N'} \in V(\Diamond)$, and simple poles at $v_1,\ldots,v_M \in V(\Lambda)$, then its \textit{divisor} is defined as \[(\omega):=\sum_{k=1}^N Q_k - \sum_{k'=1}^{N'} 2Q'_{k'}-\sum_{j=1}^M v_j.\]
\end{definition}

\begin{remark}
In the linear theory of discrete Riemann surfaces the (pointwise) product of discrete holomorphic functions is not discrete holomorphic function in general. That is also the reason why we cannot give a local definition of poles and zeroes of higher order. However, in Section~\ref{sec:branch} we merged several branch points to define one branch point of higher order. In a slightly different way, we can consider a finite subgraph $\Diamond_0\subseteq\Diamond$ that forms a simply-connected closed region consisting of $F$ quadrilaterals, where each quadrilateral is a double value of the discrete meromorphic function $f$, as one multiple value of order $F+1$. Then, $f$ takes the same value at all black vertices of $\Diamond_0$ and at all white vertices of $\Diamond_0$. If $\Diamond_0$ contains no interior vertex, both the numbers of black and of white vertices equal $F+1$, and if in addition $f$ equals zero at each black vertex, then we can interpret the $F+1$ black vertices of $\Diamond_0$ as a zero of order $F+1$.

In a similar way, double poles of discrete Abelian differentials can be merged to a pole of higher order. Unfortunately, we do not see a way how higher order poles of discrete meromorphic functions or multiple zeroes of discrete Abelian differentials can be defined.
\end{remark}

\begin{definition}
Let $D$ be a divisor. By $L(D)$ we denote the complex vector space of discrete meromorphic functions $f$ that vanish identically or whose divisor satisfies $(f)\geq D$. Similarly, $H(D)$ denotes the complex vector space of discrete Abelian differentials $\omega$ such that $\omega \equiv 0$ or $(\omega)\geq D$. The dimensions of these spaces are denoted by $l(D)$ and $i(D)$, respectively.
\end{definition}

We are now able to formulate and prove the following \textit{discrete Riemann-Roch theorem}.

\begin{theorem}\label{th:Riemann_Roch}
If $D$ is an admissible divisor on a compact discrete Riemann surface of genus $g$, then \[l(-D)=\deg D-2g+2+i(D).\]
\end{theorem}
\begin{proof}
We write $D=D_0-D_{\infty}$, where $D_0,D_{\infty}\geq 0$. Since $D$ is admissible, $D_0$ is a sum of elements of $V(\Diamond)$, all coefficients being one. Let $V_0$ denote the set of $Q\in V(\Diamond)$ such that $D_0\geq Q$.

For each $Q \in V(\Diamond)$, we fix a chart $z_Q$. As in Proposition~\ref{prop:basis}, we denote the normalized Abelian differentials of the first kind by $\omega_k^B$ and $\omega_k^W$, $k=1,\ldots,g$, these of the second kind by $\omega_Q$, $Q \in V(\Diamond)$, and these of the third kind by $\omega_{bb'}$ and $\omega_{ww'}$, $b'\neq b$ being black and $w'\neq w$ being white vertices, $b \in V(\Gamma)$ and $w \in V(\Gamma^*)$ fixed. Now, we investigate the image $H$ of the discrete exterior derivative $d$ on functions in $L(-D)$. $H$ consists of discrete Abelian differentials and only biconstant functions are in the kernel.

Let $f\in L(-D)$. Then, $df$ is a discrete Abelian differential that might have double poles at the points of $D_0$. In addition, all the residues and periods of $df$ vanish. So since the discrete Abelian differentials above form a basis by Proposition~\ref{prop:basis}, \[df=\sum_{Q\in V_0} f_Q \omega_Q\] for some complex numbers $f_Q$. Now, all black and white $b$-periods of $df$ vanish. Using Proposition~\ref{prop:periods_second} and the remark at the end of Section~\ref{sec:Abelian_differentials} on the black and white $b$-periods of $\omega_Q$, \[\sum_{Q\in V_0} f_Q\alpha_k^B(Q)=0=\sum_{Q\in V_0} f_Q\alpha_k^W(Q),\] where $\omega_k^B|_{\partial F_Q}=\alpha_k^B(Q) dz_Q$ and $\omega_k^W|_{\partial F_Q}=\alpha_k^W(Q) dz_Q$ for $k=1,\ldots,g$.

In the chart $z_{Q'}$ of a face $Q'\neq Q$, $\omega_{Q}$ can be written as $\beta_{Q}(Q')dz_{Q'}$. So if $D_{\infty}\geq 2P$, $P\in V(\Diamond)$, then $f$ has a double value at $P$ and $df|_{\partial F_P}=0$ for the corresponding face $F_P$ of $X$. Due to Lemma~\ref{lem:property_second}, \[0=\sum_{Q\in V_0} f_Q\beta_Q(P)=\sum_{Q\in V_0} f_Q\beta_P(Q).\]

Suppose that $D_{\infty}\geq v+v'$, where $v,v' \in V(\Gamma)$ or $v,v' \in V(\Gamma^*)$. By definition, $f$ has zeroes at $v,v'$, so $f(v)=f(v')$. The last equality remains true when a biconstant function is added. This yields an additional restriction to $H$. Now, using discrete Stokes' Theorem~\ref{th:stokes}, $d\omega_{vv'}$ equals $2\pi i$ when integrated over $F_v$, $-2\pi i$ when integrated over $F_{v'}$, and zero around all other vertices. Also, it follows that $\iint_{F(X)} d\left(f\omega_{vv'}\right)=0$. Writing $\omega_{vv'}|_{\partial{F_Q}}=\gamma_{vv'}(Q) dz_Q$, we observe that for $Q\in V(\Diamond)$ such that $D_0\geq Q$,
 \[\left(df\wedge\omega_{vv'}\right)|_{F_Q}=f_Q\gamma_{vv'}(Q)\frac{\pi}{2\textnormal{area}(z_Q(F_Q))}dz_Q\wedge d\bar{z}_Q,\] and $df\wedge\omega_{vv'}=0$ everywhere else. Using $d\left(f\omega_{vv'}\right)=fd\omega_{vv'}+df\wedge\omega_{vv'}$ by Theorem~\ref{th:derivation}, we obtain \[0=2\pi i\left( f(v)-f(v')\right)=\iint_{F(X)} fd\omega_{vv'}=\iint_{F(X)} d\left(f\omega_{vv'}\right)-\iint_{F(X)} df\wedge\omega_{vv'}=2\pi i\sum_{Q\in V_0}f_Q\gamma_{vv'}(Q).\]

In the case that there are more than two black (or white) vertices $v$ that satisfy $D_{\infty}\geq v$, we fix one such black (or white) vertex as $b$ (or $w$). Denote by $B_0$ and $W_0$ the sets of these black and white vertices. Then, $f$ is constant on $B_0$ and $W_0$ if and only if for any $b' \in B_0$, $b'\neq b$, and $w'\in W_0$, $w'\neq w$: \[\sum_{Q\in V_0}f_Q\gamma_{bb'}(Q)=0=\sum_{Q\in V_0}f_Q\gamma_{ww'}(Q).\] 

Consider the pmatrix $M$ whose $k$-th column is the column vector $\left(\alpha_k^B(Q)\right)_{Q \in V_0}$, whose $(g+k)$-th column is the column vector $\left(\alpha_k^W(Q)\right)_{Q \in V_0}$, and whose next columns are the column vectors $\left(\beta_P (Q)\right)_{Q \in V_0}$, $P\in V(\Diamond)$ such that $D_\infty\geq 2P$. In the case that $|B_0|\geq 2$, we add the column vectors $\left(\gamma_{bb'} (Q)\right)_{Q \in V_0}$, $b'\in B_0$ different from $b$, to $M$, and if $|W_0|\geq 2$, then we additionally add the column vectors $\left(\gamma_{ww'} (Q)\right)_{Q \in V_0}$, $w'\in W_0$ different from $w$. By our consideration above, $\sum_{Q\in V_0} f_Q \omega_Q$ is in $H$ only if the column vector $(f_Q)_{Q \in V_0}$ is in the kernel of $M^T$. Conversely, any element of the kernel is a closed discrete differential with vanishing periods, so it can be integrated to a discrete meromorphic function $f$ using Lemma~\ref{lem:multivalued}. $f$ can have poles only at $Q\in V_0$, it is biconstant on $B_0 \cup W_0$, and it has double values at all $P \in V(\Diamond)$ such that $D_{\infty} \geq 2P$. Hence, $H$ is isomorphic to the kernel of $M^T$. We obtain \[\textnormal{dim } H=\textnormal{dim } \textnormal{ker } M^T=|V_0|-\textnormal{rank } M^T=\deg D_0-\textnormal{rank } M.\]

Let us first suppose that $B_0$ and $W_0$ are both nonempty. This means that at least one zero of $f$ at a black and one at a white vertex is fixed. Therefore, $d:L(-D)\to H$ has a trivial kernel and $l(-D)=\dim H$. In addition, $\textnormal{rank } M=2g+\deg D_{\infty}-2-\textnormal{dim }\textnormal{ker } M.$ But with complex numbers $\lambda_j$, the kernel of $M$ consists of discrete Abelian differentials \[\omega=\sum_{k=1}^g \left(\lambda_k \omega_k^B+\lambda_{k+g} \omega_k^W\right)+\sum_{P: D_\infty\geq 2P} \lambda_P \omega_P+\sum_{b'\in B_0, b'\neq b}\lambda_{b'}\omega_{bb'}+\sum_{w'\in W_0, w'\neq w}\lambda_{w'}\omega_{ww'}\] such that $\omega|_{\partial {F_Q}}=0$ for any $Q \in V_0$, so the kernel is exactly $H(D)$. It follows that \begin{align*} l(-D)=\textnormal{dim } H&=\deg D_0-\textnormal{rank } M\\&=\deg D_0-2g-\deg D_{\infty}+2+\textnormal{dim }\textnormal{ker } M=\deg D-2g+2+i(D).\end{align*}

If $B_0$ or $W_0$ is empty, then we can add an additive constant to all values of $f$ at black or white vertices, respectively, still getting an element of $L(-D)$. Thus, the kernel of $d:L(-D)\to H$ is one- or even two-dimensional, when both $B_0$ and $W_0$ are empty. But $\textnormal{rank } M$ is now $2g+\deg D_{\infty}-x-\textnormal{dim }\textnormal{ker } M$ with $x=1$ or $x=0$. Again, we get $l(-D)=\deg D-2g+2+i(D).$
\end{proof}

\begin{remark}
The difference between the classical and the discrete Riemann-Roch theorem is explained by the fact that $\Lambda$ is bipartite: The space of constant functions is no longer one- but two-dimensional; instead of just $g$ $a$-periods of Abelian differentials we have $2g$, namely black and white.

Furthermore, note that the interpretation of several neighboring double values (or poles) as a multiple value (or pole) of higher order is compatible with the discrete Riemann-Roch theorem.
\end{remark}

Let us just state the following corollary for a quadrangulated flat torus that was already mentioned in \cite{BoSk12}. The proof is a consequence of the discrete Riemann-Roch Theorem~\ref{th:Riemann_Roch} and the fact that $dz$ is a discrete holomorphic differential on the torus. For the second part, one uses the decomposition of a function into real and imaginary part. For details, see the thesis \cite{Gue14}.

\begin{corollary}\label{cor:torus_poles}
Let $\Sigma=\mC/(\mZ+\mZ\tau)$ be bipartitely quadrangulated and $\im \tau >0$.
\begin{enumerate}
\item There exists no discrete meromorphic function with exactly one simple pole.
\item Suppose that in addition the diagonals of all quadrilaterals are orthogonal to each other. Then, there exists a discrete meromorphic function with exactly two simple poles at $Q,Q' \in V(\Diamond)$ if and only if the black diagonals of $Q,Q'$ are parallel to each other.
\end{enumerate}
\end{corollary}

The first part of Corollary~\ref{cor:torus_poles} does not remain true if we consider general discrete Riemann surfaces: 

\begin{proposition}\label{prop:counterexample_cauchykernel}
For any $g\geq 0$, there exists a compact discrete Riemann surface of genus $g$ such that there exists a discrete meromorphic function $f:V(\Lambda)\to\mC$ that has exactly one simple pole.
\end{proposition}
\begin{proof}
We start with any compact discrete Riemann surface $(\Sigma',\Lambda',z')$ of genus $g$ and pick one quadrilateral $Q'\in V(\Diamond')$. Now, $Q'$ is combinatorially replaced by the five quadrilaterals of Figure~\ref{fig:stamp}. We define the discrete complex structure of the central quadrilateral $Q$ by the complex number $\varrho_1$ and the discrete complex structure of the four neighboring quadrilaterals $Q_k$ by $\varrho_2$, $\re (\varrho_k)>0$. Clearly, this construction yields a new compact discrete Riemann surface $(\Sigma,\Lambda,z)$ of genus $g$. 

\begin{figure}[htbp]
\begin{center}
\beginpgfgraphicnamed{spider4}
\begin{tikzpicture}
[white/.style={circle,draw=black,fill=white,thin,inner sep=0pt,minimum size=1.2mm},
black/.style={circle,draw=black,fill=black,thin,inner sep=0pt,minimum size=1.2mm},scale=0.7]
\foreach \x in {-2,-1}
		{
		\ifthenelse{\isodd{\x}}{\node[black] (p_\x_1)  at (\x,\x) {}; \node[white] (p_\x_2)  at (\x,-\x) {}; }{\node[white] (p_\x_1)   at (\x,\x) {}; \node[black] (p_\x_2)  at (\x,-\x) {};}
		\draw (p_\x_1) -- (p_\x_2);
		}
\foreach \x in {1,2}
		{
		\ifthenelse{\isodd{\x}}{\node[black] (p_\x_1)  at (\x,\x) {}; \node[white] (p_\x_2) at (\x,-\x) {}; }{\node[white] (p_\x_1)  at (\x,\x) {}; \node[black] (p_\x_2)  at (\x,-\x) {};}
		\draw (p_\x_1) -- (p_\x_2);
		}
\foreach \x in {-2,-1,1,2}
		{\pgfmathparse{int(multiply(\x,-1))};
		\draw (p_\x_1) -- (p_\pgfmathresult_2);
		}
\foreach \x in {1}
		{\pgfmathparse{int(add(-\x,-1))};
		\draw (p_-\x_1) -- (p_\pgfmathresult_1);
		\draw (p_-\x_2) -- (p_\pgfmathresult_2);
		\pgfmathparse{int(add(\x,1))};
		\draw (p_\x_1) -- (p_\pgfmathresult_1);
		\draw (p_\x_2) -- (p_\pgfmathresult_2);
		}
\coordinate[label=center:$Q$] (phi0) at (0,0);
\coordinate[label=center:$Q_3$] (phi1) at (0,1.5);
\coordinate[label=center:$Q_2$] (phi2) at (1.5,0);
\coordinate[label=center:$Q_1$] (phi3) at (0,-1.5);
\coordinate[label=center:$Q_4$] (phi4) at (-1.5,0);
\coordinate[label=center:$w_-$] (phi5) at (0.63,-0.7);
\coordinate[label=center:$w_+$] (phi6) at (-0.61,0.64);
\coordinate[label=center:$b_-$] (phi7) at (-0.56,-0.64);
\coordinate[label=center:$b_+$] (phi8) at (0.7,0.7);

\end{tikzpicture}
\endpgfgraphicnamed
\caption{Replacement of chosen quadrilateral by five new quadrilaterals}
\label{fig:stamp}
\end{center}
\end{figure}
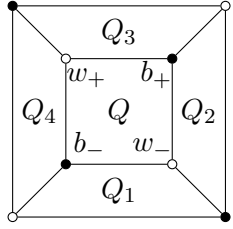

For a complex number $x\neq0$, consider the function $f:V(\Lambda)\to\mathds{C}$ that fulfills $f(b_-)=x=-f(b_+)$, $f(w_+)=i\varrho_2 x=-f(w_-)$, and $f(v)=0$ for all other vertices. Then, $f$ is a discrete meromorphic function that has exactly one simple pole, namely at $Q$.
\end{proof}


\subsection{Discrete Abel-Jacobi maps}\label{sec:Abel}

Due to the fact that black and white periods of discrete holomorphic one-forms do not have to coincide, we cannot define a discrete Abel-Jacobi map on all of $V(\Lambda)$ and $V(\Diamond)$. However, by either restricting to black vertices (and faces) or white vertices (and faces) or considering the universal covering of the compact discrete Riemann surface $(\Sigma,\Lambda,z)$ of genus $g$, we get reasonable discretizations of the Abel-Jacobi map.

\begin{definition}
Let $\omega$ denote the column vector with entries $\omega_k$, $\{\omega_1,\ldots,\omega_g\}$ being the canonical set of discrete holomorphic differentials. The $g\times g$-matrices $\Pi^B$ and $\Pi^W$ with entries $\Pi^B_{jk}:=2\int_{Bb_j}\omega_k$ and $\Pi^W_{jk}:=2\int_{Wb_j}\omega_k$ are called the \textit{black} and \textit{white period pmatrix}, respectively. Let $L$ denote the lattice $L:=\left\{Im+\Pi n | m,n \in \mathds{Z}^g\right\}$, where $I$ is the $(g\times g)$-identity pmatrix. Similarly, the lattices $L^B$ and $L^W$ with $\Pi^B$ and $\Pi^W$ instead of $\Pi$ are defined. Then, the complex tori $\mathcal{J}:=\mathds{C}^g/L$, $\mathcal{J}^B:=\mathds{C}^g/L^B$, and $\mathcal{J}^W:=\mathds{C}^g/L^W$ are the \textit{discrete}, the \textit{black}, and the \textit{white Jacobian variety}, respectively.
\end{definition}
\begin{remark}
In the notation of Section~\ref{sec:period_matrices}, $\Pi^B=\Pi^{B,W}+\Pi^{B,B}$ and $\Pi^W=\Pi^{W,W}+\Pi^{W,B}$.
\end{remark}

\begin{definition}
Let $\tilde{Q},{\tilde{Q}'}\in F(\tilde{\Lambda})$, $v\in V(\tilde{\Gamma})$, $v' \in V(\tilde{\Gamma}^*)$. Let $R$ be an oriented path on $\tilde{\Gamma}$ connecting a black vertex $b\sim\tilde{Q}$ with $v$, and let $d$ be an edge of $\tilde{X}$ parallel to the black diagonal of $\tilde{Q}$ oriented toward $b$. Lifting the discrete differentials of $\omega$ to the universal covering $(\tilde{\Sigma},\tilde{\Lambda},z \circ p)$, \[\tilde{\mathcal{A}}_{\tilde{Q}}(v):=\tilde{\mathcal{A}}^B_{\tilde{Q}}(v):=\int_{\tilde{Q}}^v \omega:=\int_d \omega + \int_{R} \omega.\] Similarly, we define $\tilde{\mathcal{A}}_{\tilde{Q}}(v'):=\tilde{\mathcal{A}}^W_{\tilde{Q}}(v'):=\int_{\tilde{Q}}^{v'} \omega$ by replacing the graph $\tilde{\Gamma}$ by $\tilde{\Gamma}^*$. Furthermore, we define $\tilde{\mathcal{A}}^B_{\tilde{Q}}({\tilde{Q}'}):=\tilde{\mathcal{A}}^B_{\tilde{Q}}(b)-\tilde{\mathcal{A}}^B_{\tilde{Q}'}(b)$ and $\tilde{\mathcal{A}}^W_{\tilde{Q}}({\tilde{Q}'}):=\tilde{\mathcal{A}}^W_{\tilde{Q}}(w)-\tilde{\mathcal{A}}^W_{\tilde{Q}'}(w)$ for a white vertex $w$ incident to $\tilde{Q}$. 
\end{definition}

\begin{remark}
Since all discrete differentials $\omega_k$ are closed, the above definitions do not depend on the choice of paths. Furthermore, $\tilde{\mathcal{A}}^B_{\tilde{Q}}: V(\tilde{\Gamma}) \cup F(\tilde{\Lambda}) \to \mathds{C}^g$ and $\tilde{\mathcal{A}}^W_{\tilde{Q}}: V(\tilde{\Gamma}^*) \cup F(\tilde{\Lambda}) \to \mathds{C}^g$ actually project to well-defined maps $\mathcal{A}^B_Q : V(\Gamma) \cup V(\Diamond) \to \mathcal{J}^B$ and $\mathcal{A}^W_Q : V(\Gamma^*) \cup V(\Diamond) \to \mathcal{J}^W$ for $Q:=p(\tilde{Q})$. These \textit{black} and \textit{white Abel-Jacobi maps} discretize the Abel-Jacobi map at least for divisors that do not include white or black vertices, respectively. Clearly, they do not depend on the base point $Q$ for divisors of degree 0.
\end{remark}

$\tilde{Q}$ can be connected with another $\tilde{Q}'\in F(\tilde{\Lambda})$ in a more symmetric way that does not depend on a choice of either black or white, using the medial graph.

\begin{definition}
Let $\tilde{Q},{\tilde{Q}'}\in F(\tilde{\Lambda})$. Let $x$ be a vertex of the face $F_{\tilde{Q}}\in F(\tilde{X})$ corresponding to $\tilde{Q}$ and $e,e'$ the two oriented edges of $F_{\tilde{Q}}$ pointing to $x$. We lift the discrete differentials of $\omega$ to the universal covering $(\tilde{\Sigma},\tilde{\Lambda},z \circ p)$. Defining $\int_{\tilde{Q}}^x \omega:=\int_e \omega/2 + \int_{e'}\omega/2$ and similarly $\int_{{\tilde{Q}'}}^{x'} \omega$ for a vertex $x'$ of $F_{\tilde{Q}'}$, \[\tilde{\mathcal{A}}_{\tilde{Q}}({\tilde{Q}'}):=\int_{\tilde{Q}}^{{\tilde{Q}'}} \omega:=\int_{\tilde{Q}}^x \omega+\int_x^{x'} \omega-\int_{{\tilde{Q}'}}^{x'} \omega.\]
\end{definition}
\begin{remark}
$\tilde{\mathcal{A}}_{\tilde{Q}}({\tilde{Q}'})$ is well-defined and does not depend on $x,x'$. In Figure~\ref{fig:contours2}, we described how a closed path on the medial graph induces closed paths on the black and white subgraph. Similarly, a ``path'' connecting $\tilde{Q}$ with ${\tilde{Q}'}$ as above induces two other paths connecting both faces, a black path just using edges of $\tilde{\Gamma}$ and a white path just using edges of $\tilde{\Gamma}^*$ (and half of a diagonal each). This construction shows that $2\tilde{\mathcal{A}}_{\tilde{Q}}({\tilde{Q}'})=\tilde{\mathcal{A}}^B_{\tilde{Q}}({\tilde{Q}'})+\tilde{\mathcal{A}}^W_{\tilde{Q}}({\tilde{Q}'})$.

Thus, $\tilde{\mathcal{A}}_{\tilde{Q}}$ defines a \textit{discrete Abel-Jacobi map} on the divisors of the universal covering $(\tilde{\Sigma},\tilde{\Lambda},z \circ p)$ and it does not depend on the choice of base point $\tilde{Q}$ for divisors of degree 0 that contain as many black as white vertices (counted with sign).
\end{remark}

\begin{proposition}\label{prop:Abel_holomorphic}
$\tilde{\mathcal{A}}_{\tilde{Q}}|_{V(\tilde{\Lambda})}$ is discrete holomorphic in each component.
\end{proposition}
\begin{proof}
Let $\tilde{Q}' \in F(\tilde{\Lambda})$ and $z_{Q'}$ be a chart of $Q'=p(\tilde{Q}')$. Then, $\omega_k|_{\partial F_{Q'}}= p_kdz_{Q'}$ for some complex numbers $p_k$. If $\tilde{b}_-,\tilde{w}_-,\tilde{b}_+,\tilde{w}_+$ denote the vertices of $Q$ in counterclockwise order, starting with a black vertex, then \begin{align*}\left(\tilde{\mathcal{A}}_{\tilde{Q}}\left(\tilde{b}_+\right)-\tilde{\mathcal{A}}_{\tilde{Q}}\left(\tilde{b}_-\right)\right)_k&=p_k\left(z_{Q'}\left(p\left(\tilde{b}_+\right)\right)-z_{Q'}\left(p\left(\tilde{b}_-\right)\right)\right),\\ \left(\tilde{\mathcal{A}}_{\tilde{Q}}\left(\tilde{w}_+\right)-\tilde{\mathcal{A}}_{\tilde{Q}}\left(\tilde{w}_-\right)\right)_k&=p_k\left(z_{Q'}\left(p\left(\tilde{w}_+\right)\right)-z_{Q'}\left(\left(\tilde{w}_-\right)\right)\right).\end{align*} Thus, the discrete Cauchy-Riemann equation is fulfilled.
\end{proof}
\begin{remark}
In such a chart $z_{\tilde{Q}'}=z_{Q'} \circ p$, $\left(\partial_\Lambda \tilde{\mathcal{A}}_{\tilde{Q}}\right)\left(\tilde{Q}'\right)=p,$ exactly as in the smooth case. In particular, the discrete Abel-Jacobi map is an injection unless there is $Q \in V(\Diamond)$ such that all discrete holomorphic differentials vanish at $Q$. By the discrete Riemann-Roch Theorem~\ref{th:Riemann_Roch}, this would imply that there exists a discrete meromorphic function with exactly one simple pole at $Q$. In contrast to the classical theory, this could happen for any genus $g$ due to Proposition~\ref{prop:counterexample_cauchykernel}.
\end{remark}

\section*{Acknowledgment}

The authors would like to thank the anonymous reviewer for his comments and suggestions, in particular for his idea of Figure~\ref{fig:cover}.

The first author was partially supported by the DFG Collaborative Research Center TRR 109, ``Discretization in Geometry and Dynamics''. The research of the second author was supported by the Deutsche Telekom Stiftung. Some parts of this paper were written at the Institut des Hautes \'Etudes Scientifiques in Bures-sur-Yvette, the Isaac Newton Institute for Mathematical Sciences in Cambridge, and the Erwin Schr\"odinger International Institute for Mathematical Physics in Vienna. The second author thanks the European Post-Doctoral Institute for Mathematical Sciences for the opportunity to stay at these institutes. The stay at the Isaac Newton Institute for Mathematical Sciences was funded through an Engineering and Physical Sciences Research Council Visiting Fellowship, Grant EP/K032208/1.


\bibliographystyle{plain}
\bibliography{Discrete_Riemann_surfaces}

\begin{thebibliography}{10}

\bibitem{BaNo07}
M.~Baker and S.~Norine.
\newblock {R}iemann-{R}och and {A}bel-{J}acobi theory on a finite graph.
\newblock {\em Adv. Math.}, 215(2):766--788, 2007.

\bibitem{BoG15}
A.I. Bobenko and F.~G\"unther.
\newblock Discrete complex analysis on planar quad-graphs.
\newblock In A.I. Bobenko, editor, {\em Advances in {D}iscrete {D}ifferential
  {G}eometry}, pages 57--132. Springer-Verlag, Berlin Heidelberg New York,
  2016.

\bibitem{BoMeSu05}
A.I. Bobenko, C.~Mercat, and Yu.B. Suris.
\newblock Linear and nonlinear theories of discrete analytic functions.
  {I}ntegrable structure and isomonodromic {G}reen's function.
\newblock {\em J. Reine Angew. Math.}, 583:117--161, 2005.

\bibitem{BoPSp10}
A.I. Bobenko, U.~Pinkall, and B.~Springborn.
\newblock Discrete conformal maps and ideal hyperbolic polyhedra.
\newblock {\em Geom. Top.}, 19:2155--2215, 2015.

\bibitem{BoSk12}
A.I. Bobenko and M.~Skopenkov.
\newblock Discrete {R}iemann surfaces: {L}inear discretization and its
  convergence.
\newblock {\em J. reine angew. Math.}, 720:217--250, 2016.

\bibitem{Bost92}
J.-B. Bost.
\newblock {I}ntroduction to compact {R}ieman surfaces, {J}acobians, and
  {A}belian varieties.
\newblock In M.~Waldschmidt, P.~Moussa, J.-M. Luck, and C.~Itzykson, editors,
  {\em {F}rom {N}umber {T}heory to {P}hysics}, pages 64--211. Springer-Verlag,
  Berlin Heidelberg, 1992.

\bibitem{Bue08}
U.~B\"ucking.
\newblock Approximation of conformal mappings by circle patterns.
\newblock {\em Geom. Dedicata}, 137:163--197, 2008.

\bibitem{ChSm11}
D.~Chelkak and S.~Smirnov.
\newblock Discrete complex analysis on isoradial graphs.
\newblock {\em Adv. Math.}, 228:1590--1630, 2011.

\bibitem{ChSm12}
D.~Chelkak and S.~Smirnov.
\newblock Universality in the 2{D} {I}sing model and conformal invariance of
  fermionic observables.
\newblock {\em Invent. Math.}, 189(3):515--580, 2012.

\bibitem{Ci12}
D.~Cimasoni.
\newblock Discrete {D}irac operators on {R}iemann surfaces and {K}asteleyn
  matrices.
\newblock {\em J. Eur. Math. Soc.}, 14:1209--1244, 2012.

\bibitem{Ci15}
D.~Cimasoni.
\newblock {K}ac-{W}ard operators, {K}asteleyn operators, and s-holomorphicity
  on arbitrary surface graphs.
\newblock {\em Ann. Inst. Henri Poincar\'e D}, 2(2):113--168, 2015.

\bibitem{CoFrLe28}
R.~Courant, K.~Friedrichs, and H.~Lewy.
\newblock {\"U}ber die partiellen {D}ifferentialgleichungen der mathematischen
  {P}hysik.
\newblock {\em Math. Ann.}, 100:32--74, 1928.

\bibitem{Du56}
R.J. Duffin.
\newblock Basic properties of discrete analytic functions.
\newblock {\em Duke Math. J.}, 23(2):335--363, 1956.

\bibitem{Du68}
R.J. Duffin.
\newblock Potential theory on a rhombic lattice.
\newblock {\em J. Comb. Th.}, 5:258--272, 1968.

\bibitem{DN03}
I.A. Dynnikov and S.P. Novikov.
\newblock Geometry of the triangle equation on two-manifolds.
\newblock {\em Moscow Math. J.}, 3(2):419--482, 2003.

\bibitem{FaKra80}
H.M. Farkas and I.~Kra.
\newblock {\em Riemann surfaces}, volume~71 of {\em Grad. Texts in Math.}
\newblock Springer, New York, 1980.

\bibitem{Fe44}
J.~Ferrand.
\newblock Fonctions pr\'eharmoniques et fonctions pr\'eholomorphes.
\newblock {\em Bull. Sci. Math. S\'er. 2}, 68:152--180, 1944.

\bibitem{Gue14}
F.~G\"unther.
\newblock {\em Discrete {R}iemann surfaces and integrable systems}.
\newblock PhD thesis, Technische Universit\"at Berlin, September 2014.
\newblock \url{http://dx.doi.org/10.14279/depositonce-4183}.

\bibitem{Is41}
R.Ph. Isaacs.
\newblock A finite difference function theory.
\newblock {\em Univ. Nac. Tucum\'an. Rev. A}, 2:177--201, 1941.

\bibitem{Ke00}
R.~Kenyon.
\newblock Conformal invariance of domino tiling.
\newblock {\em Ann. Probab.}, 28(2):759--795, 2002.

\bibitem{Ke02}
R.~Kenyon.
\newblock The {L}aplacian and {D}irac operators on critical planar graphs.
\newblock {\em Invent. math.}, 150:409--439, 2002.

\bibitem{Fe55}
J.~Lelong-Ferrand.
\newblock {\em Repr\'esentation conforme et transformations \`a int\'egrale de
  Dirichlet born\'ee}.
\newblock Gauthier-Villars, Paris, 1955.

\bibitem{Me01b}
C.~Mercat.
\newblock Discrete period matrices and related topics.
\newblock arXiv:math-ph/0111043, 2001.

\bibitem{Me01}
C.~Mercat.
\newblock Discrete {R}iemann surfaces and the {I}sing model.
\newblock {\em Commun. Math. Phys.}, 218(1):177--216, 2001.

\bibitem{Me07}
C.~Mercat.
\newblock Discrete {R}iemann surfaces.
\newblock In {\em Handbook of {T}eichm\"uller theory. Vol. I}, volume~11 of
  {\em IRMA Lect. Math. Theor. Phys.}, pages 541--575, Zurich, 2007. Eur. Math.
  Soc.

\bibitem{Me08}
C.~Mercat.
\newblock Discrete complex structure on surfel surfaces.
\newblock In {\em Proceedings of the 14th {IAPR} international conference on
  {D}iscrete geometry for computer imagery}, DGCI'08, pages 153--164, Berlin,
  Heidelberg, 2008. Springer-Verlag.

\bibitem{PP93}
U.~Pinkall and K.~Polthier.
\newblock Computing discrete minimal surfaces and their conjugates.
\newblock {\em Exper. Math.}, 2(1):15--36, 1993.

\bibitem{RSul87}
B.~Rodin and D.~Sullivan.
\newblock The convergence of circle packings to the {R}iemann mapping.
\newblock {\em J. Diff. Geom.}, 26(2):349--360, 1987.

\bibitem{Sm01}
S.~Smirnov.
\newblock Critical percolation in the plane: {C}onformal invariance,
  {C}ardy’s formula, scaling limits.
\newblock {\em C. R. Math. Acad. Sci. Paris S\'er. I}, 333(3):239--244, 2001.

\bibitem{Sm10}
S.~Smirnov.
\newblock Conformal invariance in random cluster models. {I}. {H}olomorphic
  fermions in the {I}sing model.
\newblock {\em Ann. Math.}, 172(2):1435--1467, 2010.

\bibitem{Sm10S}
S.~Smirnov.
\newblock Discrete complex analysis and probability.
\newblock In {\em Proceedings of the {I}nternational {C}ongress of
  {M}athematicians 2010 (ICM 2010)}, Vol. I: Plenary Lectures and Ceremonies,
  Vols. II-IV: Invited Lectures, pages 595--621, New Delhi, India, 2010.
  Hindustan Book Agency.

\bibitem{Ste05}
K.~Stephenson.
\newblock {\em Introduction to circle packing: {T}he theory of discrete
  analytic functions}.
\newblock Cambridge University Press, Cambridge, 2005.

\bibitem{Wi08}
S.O. Wilson.
\newblock Conformal cochains.
\newblock {\em Trans. Amer. Math. Soc.}, 360(10):5247--5264, 2008.

\end{thebibliography}

\end{document}